\documentclass[reqno,11pt, a4paper]{amsart}
\usepackage[utf8]{inputenc}
\usepackage[T2A]{fontenc}
\usepackage[russian,english]{babel}
\usepackage{xcolor}
\usepackage[colorlinks,citecolor=green!50!black,linkcolor=blue]{hyperref}
\usepackage{amssymb,amsfonts,amsmath}
\usepackage[margin=2.54cm,includeheadfoot]{geometry}
\usepackage{mathrsfs}
\usepackage{mathtools}
\mathtoolsset{showonlyrefs}
\usepackage{tikz}
\usepackage{tikz-3dplot}
\usetikzlibrary{math}
\usepackage{pgfplots}
\pgfplotsset{compat=1.11}
\usetikzlibrary{arrows.meta}
\usetikzlibrary{calc}
\usetikzlibrary{decorations.markings}

\usepackage{wrapfig}

\usepackage[color,leftbars]{changebar}

\let\div\relax
\DeclareMathOperator{\div}{div}
\DeclareMathOperator{\dive}{div}

\DeclareMathOperator{\sign}{sign}
\DeclareMathOperator{\supp}{supp}

\let\d\relax
\newcommand{\d}{\partial}
\newcommand{\eps}{\varepsilon}

\newcommand{\yu}{\boldsymbol{u}}
\newcommand{\bi}{\boldsymbol{b}}
\newcommand{\Bi}{\boldsymbol{B}}

\let\v\relax
\newcommand{\v}{\boldsymbol{v}}
\newcommand{\vi}{\boldsymbol{v}}
\newcommand{\vv}{\boldsymbol{w}}

\newcommand{\R}{\mathbb{R}}

\newcommand{\N}{\mathbb{N}}
\newcommand{\Z}{\mathbb{Z}}

\newcommand{\fhi}{\varphi}
\newcommand{\1}{\mathbf{1}}
\renewcommand{\L}{\mathscr{L}}
\newcommand{\Lip}{\operatorname{\mathrm{Lip}}}

\renewcommand{\H}{\mathscr{H}}
\DeclareMathOperator{\rest}{\llcorner}
\newcommand{\loc}{\text{\rm loc}}

\newtheorem{theorem}{Theorem}[section]
\newtheorem{proposition}[theorem]{Proposition}
\newtheorem{lemma}[theorem]{Lemma}
\newtheorem{conjecture}[theorem]{Conjecture}
\newtheorem{corollary}[theorem]{Corollary}
\newtheorem{definition}[theorem]{Definition}

\theoremstyle{definition}
\newtheorem{remark}[theorem]{Remark}

\title{On flows generated by square-integrable vector fields}

\author{N.A. Gusev}
\address[N.G.]{Moscow Institute of Physics and Technology,
	9 Institutskiy per., Dolgoprudny, Moscow Region, Russia 141700
}
\email{ngusev@phystech.su; ngusev@phystech.edu; n.a.gusev@gmail.com}
\definecolor{royalazure}{rgb}{0.0, 0.22, 0.66}
\definecolor{green}{rgb}{0,0.5,0}

\author{M.V. Korobkov}
\address[M.K.]{School of Mathematical Sciences,
Fudan University,
Shanghai 200433,
People’s Republic of China}
\address[M.K.]{Sobolev Institute of Mathematics, Acad. Koptyug pr. 4, Novosibirsk, Russia.}
\email{korob@math.nsc.ru}

\author{E.Yu.~Panov}
\address[E.P.]{%
	St. Petersburg Department of V.A. Steklov Institute of Mathematics,
	27 Fontanka,  St. Petersburg, Russia 191023%
}
\address[E.P.]{%
	Yaroslav-the-Wise Novgorod State University,
	41 B. St.-Petersburgskaya st., Veliky Novgorod, Russia 173003%
}
\email{evpanov@yandex.ru}

\author{K.Yu. Zamana}
\address[K.Z.]{Moscow Institute of Physics and Technology,
	9 Institutskiy per., Dolgoprudny, Moscow Region, Russia 141700
}
\email{zamana.kyu@phystech.edu}

\begin{document}
\date{\today}
\allowdisplaybreaks

\begin{abstract}
For square-integrable divergence-free vector field $\v$ on $\R^d$ we prove that the following properties are equivalent:
\begin{enumerate}
	\item the operator $A_0 \rho = \v \cdot \nabla \rho$ (where $\rho \in C^\infty_c(\R^d)$) is essentially skew-adjoint on $L^2(\R^d)$;
	\item square-integrable (with respect to spatial variables) generalized solutions of the continuity equation are renormalized;
	\item generalized square-integrable (with respect to spatial variables) solutions of the Cauchy problem for the corresponding continuity equation are unique both forward an backward in time.
\end{enumerate}
We also construct a compactly supported bounded divergence-free vector field $\vi\colon \R^3 \to \R^3$ for which square-integrable (with respect to spatial variables) solutions of the Cauchy problem for the corresponding continuity equation are unique forward, but not backward in time.
\end{abstract}

\maketitle
\tableofcontents

\section{Introduction}

Let $\v \in \boldsymbol{L}^2(\R^d)$ be a compactly supported vector field, which is divergence-free (in the sense of distributions).
For any $\rho \in C_c^\infty(\R^d)$ let
\begin{equation}\label{differentiation-along-v}
	A_0 \rho := \v \cdot \nabla \rho.
\end{equation}
Clearly $A_0$ is a densely defined linear operator on $L^2(\R^d)$ with the domain $D(A_0) = C_c^\infty(\R^d)$.
In \cite{Nelson_1963} (cf. \cite{Aizenman_1978}) the following conjecture was
proposed:
\begin{conjecture}\label{Nelson}
The operator $A_0$ is essentially skew-adjoint.
\end{conjecture}

One of the observations, on which Conjecture~\ref{Nelson} was based originally, is the following one: if the operator $A_0$ is essentially skew-adjoint, then there exists at most one measure-preserving flow of $\vi$. This property was stated (without a proof) in \cite[p. 288]{Aizenman_1978}. For completeness we prove it in the Appendix~\ref{a:uniqueness-of-mpf}, see Proposition~\ref{unique-flow}. 

We study the connection between the essential skew-adjointness
of the operator $A_0$ and well-posedness of the initial value problem
for the continuity equation with the vector field $\vi$:
\begin{equation}
\label{CE}
\d_t \rho + \dive (\rho \v) = 0.
\end{equation}
Existence and uniqueness of solutions of the initial value problem for the continuity equation is closely related to existence and uniqueness of flow of $\vi$, see e.g. \cite{DiPernaLions1989}, \cite{AmbrosioCrippa2008}, \cite{Gusev2018} and other papers.

Our main result is the following:
\begin{theorem}\label{main}
Let $\v \in \boldsymbol{L}^2(\R^d)$ be a compactly supported divergence-free vector field.
Let $I \subset \R$ be a nonempty bounded open interval.
Then the following three properties are equivalent:
\begin{enumerate}
	\item[(i)] the operator $A_0$ given by \eqref{differentiation-along-v} is essentially skew-adjoint;
	\item[(ii)] $\v$ has the renormalization property in the class $\rho \in L^\infty(I; L^2(\R^d))$;
	\item[(iii)] for any $\rho_0 \in L^2(\R^d)$ and any $s\in I$
	there exists a unique solution $\rho\in L^\infty(I; L^2(\R^d))$
	of \eqref{CE} such that $\rho|_s = \rho_0$.
\end{enumerate}
\end{theorem}
(The sense in which the condition $\rho|_s = \rho_0$ is understood is discussed in Section~\ref{sec:notation}. The implications (i)$\to$(ii), (ii)$\to$(iii) and (iii)$\to$(i) of Theorem~\ref{main} are proved in Sections~\ref{sec:skadj-renorm}, \ref{section-renorm-prop} and \ref{sec:uniq-skadj} respectively. In fact the implication (ii)$\to$(iii) is well-known, but we include its proof for completeness.)

Theorem~\ref{main}
extends the results obtained in
\cite{Panov2015,Panov2018} for the continuity equation with bounded vector field (see also \cite{BouchutCrippa2006}).
Recently some of these results were generalized to the case of arbitrary densely defined skew-symmetric operator $A_0$ on a Hilbert space, see \cite{Panov2025,Panov2026}.

It is interesting to note that in the special case $d=2$ uniqueness of solutions of the Cauchy problem for the continuity equation forward in time is equivalent to uniqueness backward in time (see Theorem~\ref{uniqueness-forward-and-backward}). 
However for $d=3$ we construct an autonomous compactly supported bounded divergence-free vector field which demonstrates that in general uniqueness forward in time is not equivalent to uniqueness backward in time.
An example of non-autonomous vector field with such properties was given in \cite[Remark 3.6]{BouchutCrippa2006}.
We refer to Section~\ref{unequivalence} for more precise statements of these results. 

\section{Notation and preliminaries}\label{sec:notation}
Let $|\cdot|$ denote the Euclidean norm on $\R^d$, $d\in \N$.
For any $x\in \R^d$ and any $r>0$ let $B_r(x):= \{y\in \R^d : |y-x|<r\}$
denote the open ball with the radius $r$.
Let $\|\cdot\|_p$ denote the standard $L^p$ norm; in particular, let $\|\cdot\| := \|\cdot\|_2$. Given a set $E$, let $\1_E$ denote the indicator of the set $E$.
For any $p \in [1,+\infty]$ and any measurable set $E \subset \R^d$
we will denote with $L^p(E)$ the set of all measurable functions $f\colon E\to \R$ with $\|f\|_p < +\infty$.

For any $p,q\in[1,+\infty]$ and any (open or closed) interval $I\subset \R$ let $L^q(I; L^p(\R^d))$ denote the set
of all measurable functions $f\colon I\times \R^d \to \R$ such that
the function $t\mapsto \int_{\R^d} \|f(t, \cdot)\|_p \, dt$ belongs to $L^q(I)$.
Let $C_w(I; L^2(\R^d))$ denote the space of weakly continuous functions $f\colon I \to L^2(\R^d)$, i.e. $f\in C_w(I; L^2(\R^d))$ if and only if for any $\fhi \in L^2(\R^d)$ the function $t\mapsto \int_{\R^d} f(t) \fhi \, dx$ belongs to $C(I)$. 

Given a time-dependent vector field $\v\colon I\times \R^d \to \R^d$
from the class $\boldsymbol{L}^1_\loc(I\times \R^d)$
let $\dive \v \equiv \dive_x \v$ denote the (distributional) divergence of $\v$ with respect to
the spacial coordinates $x\in \R^d$. Similarly, for a time-dependent
scalar field $\rho \colon I \times \R^d \to \R$
we will denote with $\nabla \rho \equiv \nabla_x \rho$ the gradient of $\rho$
with respect to the spacial variables $x\in \R^d$.

Given a Borel map $f\colon A \to B$, where $A \subset \R^n$ and $B\subset \R^m$
($m,n\in \N$) are some Borel sets, for any Borel measure $\mu$ on $A$
let $f_\# \mu$ denote the image of the measure $\mu$ under the map $f$,
i.e. the measure defined by $(f_\# \mu)(E) = \mu(f^{-1}(E))$
for all Borel sets $E \subset B$.

Let $\L^d$ denote the Lebesgue measure on $\R^d$.
For measurable sets $E\subset \R^d$ we will also
denote $|E| = \L^d(E)$.

Given a nonempty open interval $I \subset \R$, $n\in \N$ and a normed space $(X, \|\cdot\|)$ let $C_c^n(I; X)$ (resp. $C_c(I; X)$) denote the set of all compactly supported functions $f\in C^n(I; X)$ (resp. $f\in C(I; X)$)
endowed with the norm of $C^n(I; X)$ (resp. $C(I;X)$).
%TODO: finite norm
As usual, $C^\infty(I; X) = \bigcap_{n=1}^\infty C^n(I; X)$
and $C^\infty_c(I; X)$ is the set of all functions from $C^\infty(I; X)$ which have compact support.

\begin{lemma}\label{dense-subspace}
	Suppose that $X_0$ is a dense linear subspace of a normed space $(X, \|\cdot\|_X)$.
	Then for any $u\in C^1_c(\R; X)$ with $\supp u \subset (a,b)$ and for any $\eps >0$
	there exists a function 
	$v$ of the form
	\begin{equation*}
		v(t) = \sum_{k=1}^n \psi_k(t) v_k, 
	\end{equation*}
	where $v_k \in X_0$ and $\psi_k \in C^\infty_c(\R)$, %are 
	such that $\supp \psi_k\subset (a,b)$ and
	\begin{equation*}
		\|u-v\|_{C^1(\R; X)} < \eps.
	\end{equation*}
\end{lemma}

\begin{proof}
Replacing, if necessary, $X$ with its completion (see e.g. \cite[Proposition 6.4.6]{BS2020}), without loss of generality we may assume that $X$ is Banach.
Let us fix $\alpha, \beta \in \R$ such that $\supp u \subset [\alpha, \beta] \subset (a,b)$.
Without loss of generality for simplicity we may assume that $\alpha = 0$ and $\beta = 1$.
Consider the convolution
\begin{equation*}
	u_\delta(t) := (u * \omega_\delta)(t) = \int_0^1 u(s) \omega_\delta(t-s) \, ds.
\end{equation*}
of $u$ with the standard mollification kernel $\omega_\delta$.
Let $\eps > 0$.
Then $u_\delta' = u' * \omega_\delta$ and for sufficiently small $\delta > 0$ we have $[-\delta, 1+\delta] \subset[a,b]$, $\supp u_\delta \subset (a,b)$ and 
\begin{equation*}
	\sup_{t\in \R}(\|u_\delta(t)-u(t)\|_X +\|u_\delta'(t) - u'(t)\|_X) < \eps.
\end{equation*}
Let us fix such $\delta$ and for $n\in\N$ consider the approximation $f_n(t)$ of the integral $u_\delta(t)$ given by
\begin{equation*}
	f_n(t) = \sum_{k=1}^n \int_{(k-1)/n}^{k/n} u(k/n) \omega_\delta(t-s) \, ds .
\end{equation*}
For all $t\in \R$ we have
\begin{equation*}
	\|f_n(t) - u_\delta(t)\|_X \le \|\omega_\delta\|_\infty \sum_{k=1}^n \int_{(k-1)/n}^{k/n} \|u(k/n) - u(s)\|_X \, ds.
\end{equation*}
Similarly, since
\begin{equation*}
	u_\delta'(t) = \int_0^1 u(s) (\omega_\delta)'(t-s) \, ds,
\end{equation*}
we have the estimate
\begin{equation*}
	\|f_n'(t) - (u_\delta)'(t)\|_X \le \|(\omega_\delta)'\|_\infty \sum_{k=1}^n \int_{(k-1)/n}^{k/n} \|u(k/n) - u(s)\|_X \, ds.
\end{equation*}
Hence by uniform continuity of $u$ for all sufficiently large $n\in \N$
we have
\begin{equation*}
	\sup_{t\in \R}(\|f_n(t) - u_\delta(t)\|_X + \|f_n'(t) - (u_\delta)'(t)\|_X) < \eps.
\end{equation*}
By the estimates obtained above 
\begin{equation*}
	\|f_n - u\|_{C^1(\R; X)} \le \|f_n - u_\delta\|_{C^1(\R; X)} + \|u_\delta - u\|_{C^1(\R; X)} < 2 \eps
\end{equation*}
for all sufficiently large $n$. Let us fix such $n$ and for all $k\in \overline{1,n}$ let
us choose $v_k \in X_0$ such that
\begin{equation*}
	\|v_k - u(k/n)\|_X < \eps \cdot \left(\int_{-1}^1 \omega_\delta(t) \, dt + \int_{-1}^1 |(\omega_\delta)'(t)| \, dt \right)^{-1}.
\end{equation*}
Let $\psi_k(t) := \int_{(k-1)/n}^{k/n} \omega_\delta(t-s) \, ds$
and
\begin{equation*}
	v(t) := \sum_{k=1}^n v_k \psi_k(t).
\end{equation*}
The functions $\psi_k(t) = \1_{((k-1)/n, \ k/n)} * \omega_\delta$ belong to $C_c^\infty(\R)$, and clearly $\supp \psi_k \subset [-\delta, 1+\delta]$.
Since $f_n(t) = \sum_{k=1}^n u(k/n) \psi_k(t)$, we have
\begin{equation*}
	\sup_{t\in \R} \|v(t) - f_n(t)\|_X \le \eps \int_{-1}^1 \omega_\delta(t) \, dt
	\cdot \left(\int_{-1}^1 \omega_\delta(t) \, dt + \int_{-1}^1 |(\omega_\delta)'(t)| \, dt \right)^{-1}
\end{equation*}
and similarly
\begin{equation*}
	\sup_{t\in \R} \|v'(t) - f_n'(t)\|_X \le \eps \int_{-1}^1 |(\omega_\delta)'(t)| \, dt
	\cdot \left(\int_{-1}^1 \omega_\delta(t) \, dt + \int_{-1}^1 |(\omega_\delta)'(t)| \, dt \right)^{-1}
\end{equation*}
Hence $\|v - f_n\|_{C^1(\R; X)} < \eps$. Combining all the estimates obtained above, we conclude that
$\|v - u\|_{C^1(\R; X)} \le \|v - f_n\|_{C^1(\R; X)} + \|f_n - u_\delta\|_{C^1(\R; X)} + \|u_\delta - u\|_{C^1(\R; X)} < 3\eps$.
\end{proof}

Any function $f\colon \R^{d+1} \to \R$ can be identified with a function
$f\colon \R\to \R^{(\R^d)}$ given by
$f(x_1)(x_2, \ldots, x_{d+1}) := f(x_1, \ldots, x_{d+1})$, and vice versa.
With this convention the following inclusion holds:
\begin{equation}\label{Cinf-C1C1}
	C^\infty_c(\R^{d+1}) \subset C^1_c(\R; C^1_c(\R^d)) %\subset C^1_c(\R^{d+1}).
\end{equation}
In view of Lemma~\ref{dense-subspace} this inclusion is dense (with respect to the corresponding $C^1$-norms).

The inclusion \eqref{Cinf-C1C1} can be partially inverted: $C^1_c(\R; C^1_c(\R^{d})) \subset C^1(\R^{d+1})$.
But a function from the former space is not necessarily compactly supported, viewed as an element of the latter one. For example one can consider $f(t,x) = t^2 a(t) b(tx)$,
where $a\in C^1_c(\R)$, $a(0)\ne 0$, $b\in C^1_c(\R^d)$, $b \not \equiv 0$.
The same example shows that $C^1_c(\R; C^1_c(\R^{d}))$ is not necessarily a subset of $L^2(\R^{d+1})$.

\subsection{Weak solutions of continuity equation}

Let $I\subset \R$ be a nonempty open interval.
Let $\v \colon I\times \R^d \to \R^d$ be a locally integrable time-dependent vector field.

\begin{definition}
A function $\rho \in L^1_\loc(I\times \R^d)$ is a (weak) solution of
\eqref{CE}
if $\rho \v \in \boldsymbol{L}^1_\loc(I\times \R^d)$ and
for any $\Phi \in C^1_c(I\times \R^d)$ the following holds:

\begin{equation}
\label{CE-D'}
\int_I \int_{\R^d} \rho(t,x) \d_t \Phi(t,x) \, dx \,dt
+ \int_I \int_{\R^d} \rho \v \cdot \nabla_x \Phi(t,x) \, dx \,dt = 0.
\end{equation}
\end{definition}

%\begin{remark}\label{equivalent-def-of-weak-sol}
%Using test functions of the form $\Phi(t,x) = \psi(t) \fhi(x)$,
%where $\psi \in C_c^1(I)$ and $\fhi \in C_c^1(\R^d)$,
%it is easy to see that $\rho\in L^1_\loc(I\times \R^d)$ solves \eqref{CE} if and only if
%for any $\fhi \in C_c^1(\R^d)$ the function $t \mapsto \int_{\R^d} \rho(t,x) \fhi(x) \, dx$ belongs to the Sobolev space $W^{1,1}_\loc(I)$ and
%\begin{equation}\label{CE-sep-var}
%	\d_t \int_{\R^d} \rho(t,x) \fhi(x) \, dx = \int_{\R^d} \rho(t, x) \v(t, x) \cdot \nabla \fhi(x) \, dx
%\end{equation}
%in the sense of distributions. Indeed, by Fubini theorem
%the function
%$$
%t\mapsto \int_{\R^d} \rho(t, x) \v(t, x) \cdot \nabla \fhi(x) \, dx
%$$
%belongs to $L^1_\loc(I)$.
%\end{remark}
%Here $C^1_c(\R^n)$ denotes the set of all compactly supported continuously differentiable functions $f\colon \R^n \to \R$.

\begin{proposition}\label{prop-CE-int}\label{equivalent-def-of-weak-sol}
Let $\rho \in L^1_\loc(I\times \R^d)$. Then the following are equivalent:
\begin{enumerate}
	\item $\rho$ solves \eqref{CE};
	\item for any $\fhi \in C_c^1(\R^d)$ the function $t \mapsto \int_{\R^d} \rho(t,x) \fhi(x) \, dx$ belongs to the Sobolev space $W^{1,1}_\loc(I)$ and
	\begin{equation}\label{CE-sep-var}
		\d_t \int_{\R^d} \rho(t,x) \fhi(x) \, dx = \int_{\R^d} \rho(t, x) \v(t, x) \cdot \nabla \fhi(x) \, dx
	\end{equation}
	in the sense of distributions (or, equivalently, for a.e. $t$);
	\item there exists a negligible set $N \subset I$ such that
	for all $s, t \in I\setminus N$ and for all $\fhi \in C^1_c(\R^d)$ % we have
	\begin{equation}\label{CE-int}
		\int_{\R^d} \rho(t,x) \fhi(x) \, dx = \int_{\R^d} \rho(s,x) \fhi(x) \, dx + \int_s^t \int_{\R^d} \rho(\tau, x) \v(\tau, x) \cdot \nabla \fhi(x) \, dx \, d\tau.
	\end{equation}
\end{enumerate}
\end{proposition}

\begin{proof}
\emph{$(1)\Rightarrow(2)$.}
By Fubini theorem the functions
$$
f_\fhi(t) :=\int_{\R^d} \rho(t,x) \fhi(x) \, dx, 
\quad
g_\fhi(t) :=  \int_{\R^d} \rho(t, x) \v(t, x) \cdot \nabla \fhi(x) \, dx
$$
belong to $L^1_\loc(I)$.
Then using test functions of the form $\Phi(t,x) = \psi(t) \fhi(x)$ (where $\psi \in C^1_c(I)$ and $\fhi \in C^1_c(\R^d)$) in \eqref{CE-D'}, it is easy to see that $(f_\fhi)' = g_\fhi$ in the sense of distributions.
Then $f_\fhi \in W^{1,1}_\loc(I)$. 
By the well-known property of the Sobolev functions the equality $(f_\fhi)' = g_\fhi$ holds in the sense of distributions if and only if it holds for a.e. $t$.
Thus \eqref{CE-sep-var} holds.

\emph{$(2)\Rightarrow(1)$.}
By \eqref{CE-sep-var} (in the sense of distributions) the equality \eqref{CE-D'} holds for all test functions of the form $\Phi(t,x) = \psi(t) \fhi(x)$, where $\psi \in C^1_c(I)$ and $\fhi \in C^1_c(\R^d)$.
In view of \eqref{Cinf-C1C1} and Lemma~\ref{dense-subspace} any $\Phi\in C^\infty_c(I\times \R^d)$ can be approximated by a finite linear combination of such functions in $C^1_c(I\times \R^d)$, and passing to the limit we conclude that \eqref{CE-D'} holds for any such~$\Phi$.
It remains to use the density of $C^\infty_c(I\times\R^d)$ in $C^1_c(I\times\R^d)$.

\emph{$(2)\Rightarrow(3)$.}
By the well-known property of the Sobolev space $W^{1,1}_\loc(I)$ there exists a negligible set
$N_\fhi \subset I$ such that for all ${s,t \in I \setminus N_\fhi}$ we have
$f_\fhi(t) - f_\fhi(s) = \int_s^t (f_\fhi)' (\tau) \, d\tau$.
Let $\mathcal F$ denote a countable set of functions $\fhi \in C_c^1(\R^d)$
such that $\mathcal F$ is dense in $C_c^1(\R^d)$. %(with respect to the corresponding $C^1$ norm)
The set $N:= \bigcup_{\fhi \in \mathcal F} N_\fhi$ is negligible,
being a countable union of negligible sets. By construction of $N$,
for all $s,t \in I\setminus N$ the equality \eqref{CE-int} holds for all
$\fhi \in \mathcal F$. Since any $\fhi \in C^1_c(\R^d)$ can be approximated by a sequence of $\{\fhi_n\}_{n\in \N} \subset \mathcal F$, one can write \eqref{CE-int} with $\fhi$ replaced by $\fhi_n$, pass to the limit as $n\to \infty$, and obtain \eqref{CE-int} with $\fhi$.

\emph{$(3)\Rightarrow(2)$.}
Since $g_\fhi \in L^1_\loc(I)$, the function $t\mapsto \int_s^t g_\fhi(\tau) \, d\tau$ belongs to $W^{1,1}_\loc(I)$. Then by \eqref{CE-int} the function $f_\fhi$ belongs to $W^{1,1}_\loc(I)$ and satisfies \eqref{CE-sep-var}. 
\end{proof}

In order to give a meaning to the initial condition in the Cauchy problem for the continuity equation, let us introduce the following definition:

\begin{definition}\label{def-trace}
Let $\rho \in L^1_\loc(I\times \R^d)$ be a weak solution of \eqref{CE}.
Let $N\subset I$ be given by Proposition~\ref{prop-CE-int} and let $s\in I\setminus N$.
Then for any $t\in \overline{I}$ the distribution
\begin{equation}\label{trace-def}
\langle \rho|_t, \fhi \rangle := \int_{\R^d} \rho(s,x) \fhi(x) \, dx + \int_s^t \int_{\R^d} \rho(\tau, x) \v(\tau, x) \cdot \nabla \fhi(x) \, dx \, d\tau,
\qquad \fhi \in \mathscr{D}(\R^d),
\end{equation}
is called the trace of $\rho$ at time $t$.
\end{definition}

\begin{remark}\label{rem-trace-and-solution}\label{rem-trace-on-closure}
By \eqref{CE-int}
for almost all $t$ the distribution $\rho|_t$ coincides with the locally integrable function $\rho(t, \cdot)$. Furthermore, by \eqref{CE-int} the distribution $\rho|_t$ does not depend on the choice of $s$, and the distribution-valued function $t \mapsto \rho|_t$ is continuous on $\overline{I}$ (i.e. for any $\fhi \in \mathscr{D}(\R^d)$
the function $t\mapsto \langle\rho|_t, \fhi\rangle$ is continuous on $\overline{I}$).
\end{remark}

In general the distribution $\rho|_t$ may be singular for some values of $t$, see e.g. \cite[Remark~1.11]{BonicattoPhD}.
However for large classes of weak solutions the regularity of $\rho|_t$
is easily improved.

\begin{proposition}\label{weakly-cont-version}
Suppose that $\rho\in L^\infty(I; L^p(\R^d))$ solves \eqref{CE},
where $1 < p \le +\infty$. Then for all $t\in \overline{I}$ we have $\rho|_t \in L^p(\R^d)$ and the function $t\mapsto \rho|_t$ is weak* continuous in $L^p(\R^d)$.
Furthermore, for all $t, s\in \overline{I}$ we have
\begin{equation}\label{CE-int-by-parts}
\langle \rho|_t, \fhi \rangle = \langle \rho|_s,  \fhi \rangle + \int_s^t \int_{\R^d} \rho(\tau, x) \v(\tau, x) \cdot \nabla \fhi(x) \, dx \, d\tau,
\qquad \fhi \in  C_c^1(\R^d),
\end{equation}
\end{proposition}

\begin{proof}
By assumption there exists a negligible set $N_1 \subset I$
such that
\begin{equation}
\sup_{t\in I \setminus N_1} \|\rho(t, \cdot)\|_p \le C < \infty.
\label{uniform-boundedness-of-rho-at-tn}
\end{equation}
Let $N_2\subset I$ be a negligible set given by Proposition~\ref{prop-CE-int}
and let $N := N_1 \cup N_2$.
For any $t\in \overline{I}$ we can choose a sequence $t_n \in I \setminus N$ converging to $t$ as $n\to \infty$. By Proposition \ref{prop-CE-int} for all $\fhi \in C_c^1(\R^d)$
there exists $\langle \rho|_t, \fhi \rangle = \lim_{n\to \infty} \int_{\R^d} \rho(t_n, x) \fhi(x)\, dx$. By \eqref{uniform-boundedness-of-rho-at-tn}
\begin{equation*}
|\langle \rho|_t, \fhi \rangle| \le C \|\fhi\|_q,
\end{equation*}
where $q=p/(p-1)$. Since $C_c^1(\R^d)$ is dense in $L^q(\R^d)$,
the functional $\rho|_t$ extends to a bounded linear functional on $L^q(\R^d)$.
Since $L^p$ is dual of $L^q$, %by Riesz theorem 
there exists unique $r_t\in L^p(\R^d)$ such that $\langle \rho|_t, \fhi \rangle = \int_{\R^d} r_t(x) \fhi(x) \, dx$ for all $\fhi \in L^q(\R^d)$.
The weak* continuity follows from the uniform boundedness of $\|r_t\|_p$.

Ultimately, \eqref{CE-int-by-parts} can be obtained by subtracting
\eqref{trace-def} from \eqref{trace-def} with different values of $t$.
\end{proof}

For a.e. $t\in I$ by Remark~\ref{rem-trace-and-solution} we have $\rho|_t = \rho(t,\cdot)$. This motivates the following definition:
\begin{definition}\label{def-weakly-cont-version}
	Under the assumptions of the Proposition~\ref{weakly-cont-version}, the map $t\mapsto \rho|_t$ is called the \emph{weakly continuous version of $\rho$}.
\end{definition}
In particular, if $t\mapsto \rho(t, \cdot)$ is weakly* continuous on $\overline{I}$, then $\rho(t,\cdot) = \rho|_t$ for \emph{all $t\in \overline{I}$}.

Much more general version of Proposition~\ref{weakly-cont-version}
can be found e.g. in \cite[Lemma 8.1.2]{AGS}.
We refer to \cite{ChenFrid1999} for more general notion of trace and related results.

Sometimes the following observation is useful (see also \cite[Theorem 2.3.3]{CrippaPhD}):

\begin{remark}\label{extension-with-zero}
	Suppose that $I=[0,+\infty)$ and $\rho_0 \in L^p(\R^d)$.
	A function $\rho \in L^\infty(I; L^p(\R^d))$ solves~\eqref{CE}
	with the initial condition $\rho|_0 = \rho_0$
	if and only if
	\begin{equation}
		\tilde \rho(t, \cdot) =
		\begin{cases}
			\rho(t, \cdot), & t > 0 \\
			\rho_0(\cdot), & \text{otherwise}
		\end{cases}
		\qquad \text{and} \qquad
		\tilde \vi(t, \cdot) =
		\begin{cases}
			\vi(t, \cdot), & t > 0 \\
			0, & \text{otherwise}.
		\end{cases}
	\end{equation}
	solve \eqref{CE} in $\mathscr{D}'(\R\times \R^{d})$.
\end{remark}

Existence of weak solutions to the Cauchy problem for the continuity equation
was established in \cite[Proposition II.1]{DiPernaLions1989}.
We will not need the most general setting, and therefore the proof can be simplified:

\begin{theorem}\label{existence-of-weak-solutions}
Let $I=(0,T),$ where $T>0$.
Suppose that $\v \in \boldsymbol{L}^2(\R^d)$ is divergence-free vector field.
Then for any $s\in \overline{I}$ and any $\rho_0 \in L^2(\R^d)$ there exists a weak solution
${\rho \in L^\infty(I;L^2(\R^d))} \cap C_w(\overline{I}; L^2(\R^d))$ of the continuity equation
such that $\rho|_s = \rho_0$ and for all $t\in \overline I$
\begin{equation*}
\|\rho(t,\cdot)\| \le \|\rho_0\|.
\end{equation*}
\end{theorem}

\begin{proof}
Let $\omega\in C_c^\infty(\R^d)$ denote the standard mollification kernel,
$\omega_\eps(x):=\eps^{-d} \omega(x/\eps)$.
Let $\v_\eps := \v * \omega_\eps$, denote the convolution of $\v$ with $\omega_\eps$.
By the well-known properties of convolution
$\v_\eps \in \boldsymbol{L}^2(\R^d) \cap \boldsymbol{L}^\infty(\R^d) \cap \boldsymbol{C}^\infty(\R^d)$,
$\dive \v_\eps = 0$ and $\|\v_\eps - \v\|_2 \to 0$ as $\eps \to 0$.
Also let $\rho_0^\eps := \rho_0 * \omega_\eps$.
(We will tacitly assume that $\eps$ is set to some
infinitesimal sequence of strictly positive numbers.)

Since $\v_\eps$ is smooth and divergence-free, the Cauchy problem
\begin{equation*}
\left\{
\begin{aligned}
&\d_t \rho^\eps + \dive(\rho^\eps \v_\eps) = 0 \qquad\text{on } I \times \R^d, \\
&\rho^\eps(s, \cdot) = \rho_0^\eps \qquad\text{on }\R^d
\end{aligned}
\right.
\end{equation*}
has a unique smooth solution $\rho^\eps=\rho^\eps(t,x)$, constructed by the method of characteristics:
\begin{equation*}
\rho^\eps(t,y) = \rho^\eps_0(X^\eps(s-t, y)),
\end{equation*}
where $X^\eps=X^\eps(t,x)$ is the flow of $\v_\eps$:
\begin{equation*}
\left\{
\begin{aligned}
& \d_t X^\eps(t,x) = \v_\eps(X^\eps(t,x)), \\
& X^\eps(0,x) = x
\end{aligned}
\right.
\end{equation*}
for all $x\in \R^d$ and $t\in \R$.
Since $\v_\eps$ is smooth and bounded,
the flow $X^\eps$ exists and is unique (by the classical Cauchy--Lipschitz theory).

Since $\v$ is divergence-free, the flow $X^\eps(t, \cdot)$ preserves the Lebesgue measure, i.e. for all $t$ we have $X^\eps(t, \cdot)_\# \L^d = \L^d$.
Then for all $t$ we have
\begin{equation}
\label{norm-of-approximate-solution}
\|\rho^\eps(t, \cdot)\|_2 = \|\rho^\eps_0\|_2.
\end{equation}
Hence the sequence $\rho^\eps$ is uniformly bounded in $L^\infty(I; L^2(\R^d))$
and, consequently, in $L^2(I \times \R^d)$.
Since norm-bounded subsets of Hilbert spaces are weakly sequentially compact,
we can extract a subsequence (without relabeling) such that $\rho^\eps$
converges weakly in $L^2(I\times \R^d)$ to some $\rho \in L^2(I\times \R^d)$.
Since $\v_\eps$ converges to $\v$ strongly in $\boldsymbol{L}^2(\R^d)$,
passing to the limit in \eqref{CE-D'}
we deduce that $\rho$ solves the continuity equation.

For all $t\in \overline{I}$ the approximate solution $\rho^\eps$ for all $\fhi\in C^1_c(\R^d)$ satisfies
\begin{equation}
\label{CE-approx-int}
\int_{\R^d} \rho^\eps(t,x) \fhi(x)\, dx = \int_{\R^d} \rho^\eps_0(x) \fhi(x) \, dx + \int_s^t \int_{\R^d} \rho^\eps(\tau, x) \v_\eps(x) \cdot \nabla \fhi(x) \, dx \, d\tau
\end{equation}
Passing in this equality to the limit as $\eps \to 0$
we obtain that for all $t \in \overline{I}$ and for all $\fhi \in C^1_c(\R^d)$
the sequence $\{\int_{\R^d} \rho^\eps(t,x) \fhi(x)\, dx\}_{\eps}$ converges.
By \eqref{norm-of-approximate-solution}
the $L^2$-norms of $\rho^\eps(t, \cdot)$ are uniformly bounded,
hence for all $t\in \overline{I}$ the sequence $\{\rho^\eps(t,\cdot)\}$
converges weakly in $L^2(\R^d)$.
Let $r_t\in L^2(\R^d)$ denote the corresponding weak limit.
Passing to the limit in \eqref{CE-approx-int} we obtain that
\begin{equation}
\label{CE-D'-partial}
\int_{\R^d} r_t(x) \fhi(x)\, dx = \int_{\R^d} \rho_0(x) \fhi(x) \, dx + \int_s^t \int_{\R^d} \rho(\tau, x) \v(x) \cdot \nabla \fhi(x) \, dx \, d\tau
\end{equation}
holds for \emph{all} $t\in \overline{I}$ and for all $\fhi \in C^1_c(\R^d)$. 
Since the $L^2$-norms of $r_t$ are uniformly bounded, by~\eqref{CE-D'-partial} the map $t\mapsto r_t$ is weakly continuous (similarly to Proposition~\ref{weakly-cont-version}).

For any $\psi\in C^1_c(I)$ and any $\fhi \in C^1_c(\R^d)$
by Fubini theorem the following equality holds:
\begin{equation}
\label{Fubini-trick}
\int_I \int_{\R^d} \rho^{\eps}(t,x) \fhi(x) \, dx \, \psi(t) \, dt
= \iint_{I\times \R^d} \rho^{\eps}(t,x) \fhi(x) \psi(t) \, dt \, dx.
\end{equation}
Since for all $t\in I$ and $\eps>0$ by \eqref{norm-of-approximate-solution} we have
$\left|\int_{\R^d}\rho^\eps(t,x) \fhi(x) \, dx\right| \le \|\rho_0\|_2 \cdot \|\fhi\|_2$,
by the dominated convergence the left-hand side of \eqref{Fubini-trick}
converges to $\int_I \int_{\R^d} r_t(x) \fhi(x) \, dx \, \psi(t) \, dt$.
On the other hand, by weak convergence of $\rho^\eps$ to $\rho$
in $L^2(I\times \R^d)$ the right-hand side of \eqref{Fubini-trick}
converges to $\iint \rho(t,x) \fhi(x) \psi(t) \, dt \, dx$.
Hence, passing to the limit in \eqref{Fubini-trick} and applying Fubini theorem
once again we obtain
\begin{equation*}
\int_I \int_{\R^d} r_t(x) \fhi(x) \, dx \, \psi(t) \, dt
= \int_I \int_{\R^d} \rho(t,x) \fhi(x) \, dt \, \psi(t) \, dx.
\end{equation*}
By arbitrariness of $\fhi$ and $\psi$ this implies that
\begin{equation*}
r_t = \rho(t,\cdot) \qquad \text{a.e. on } \R^d
\end{equation*}
for a.e. $t\in I$.
Then the function $\widetilde{\rho}(t, x) := r_t(x)$ belongs to $L^\infty(I;L^2(\R^d)) \cap C_w(\overline{I}; L^2(\R^d))$ and by \eqref{CE-D'-partial} for all $t$ we have
\begin{equation*}
	\int_{\R^d} \widetilde{\rho}(t,x) \fhi(x)\, dx = \int_{\R^d} \rho_0(x) \fhi(x) \, dx + \int_s^t \int_{\R^d} \widetilde{\rho}(\tau, x) \v(x) \cdot \nabla \fhi(x) \, dx \, d\tau.
\end{equation*}
Hence $\widetilde{\rho}|_s = \rho_0$.
Thus $\widetilde{\rho}(t, x) := r_t(x)$ belongs to $L^\infty(I;L^2(\R^d)) \cap C_w(\overline{I}; L^2(\R^d))$ and solves~\eqref{CE} with the initial condition $\rho_0$.
\end{proof}

In particular, for $s=0$ Theorem~\ref{existence-of-weak-solutions}
provides existence of weak solutions with the initial condition $\rho_0$.

\subsection{Renormalization property}\label{section-renorm-prop}

\begin{definition}
	We say that a function $\beta\colon\R\to\R$ is
	a renormalizer (or admissible function, cf.
\cite[II.3]{DiPernaLions1989})
	if it satisfies the following conditions:
	
	\begin{enumerate}
		\item $\beta\in C^1(\R)$,
		
		\item $\beta'\in L^{\infty}(\R)$,
		
		\item $\beta(0)=0$.
	\end{enumerate}
\end{definition}

\begin{remark}\label{renorm_is_L2}
	Let $E \subset \R^n$ be a measurable set.
	If $\beta$ is a renormalizer, then $\beta\circ u\in L^2(E)$ for any $u\in L^2(E)$. Indeed, $\beta\circ u$ is Lebesgue measurable since $\beta$ is continuous and $u$ is measurable, while $\|\beta\circ u\|\le\|\beta'\|_{\infty}\|u\|<\infty$, which follows from the mean value theorem and $\beta(0)=0$.
	Similarly, if $u \in L^1_\loc(E)$ then $\beta\circ u \in L^1_\loc(E)$.
\end{remark}

\begin{definition}\label{def-renorm}
Let $\v\in\boldsymbol{L}^1_\loc(I\times \R^d)$ be a (time-dependent) divergence-free vector field.
Let $I \subset \R$ be a nonempty open interval.
We say that $\v$ has the renormalization property in the class $X \subset L^1_\loc(I\times \R^d)$
if for any solution $\rho \in X$ of \eqref{CE} for any admissible function $\beta$
the function $\beta\circ \rho$ solves \eqref{CE}, i.e.
\begin{equation}\label{renorm}
\d_t \beta(\rho) + \dive (\beta(\rho) \v) = 0.
\end{equation}
\end{definition}

% !TeX root = skew-adjoint.tex

Note that usually in the definition of renormalization property one requires that the initial condition for $\beta(\rho)$ is $\beta(\rho_0)$, where $\rho_0$ is the initial condition for $\rho$.
However, for autonomous~$\vi$ such property will follow from the weaker Definition~\ref{def-renorm}, as we will see.

Renormalization property allows to deduce uniqueness of weak solutions of \eqref{CE}
and to improve regularity of the traces given by Proposition~\ref{weakly-cont-version}:

\begin{theorem}\label{renormalized-solutions}
	Let $I\subset \R$ be an open interval.
	Let $\v \in \boldsymbol{L}^2(\R^d)$ be a compactly supported divergence-free
	vector field and suppose that $\v$ has the renormalization property in $L^\infty(I; L^2(\R^d))$.
	Then for any $s\in \overline{I}$ and any $\rho_0 \in L^2(\R^d)$ the continuity
	equation has at most one weak solution $\rho \in L^\infty(I; L^2(\R^d))$
	such that $\rho|_s = \rho_0$.
\end{theorem}

\begin{proof}
	First let us suppose that $s$ is an interior point of $I$.
	By linearity of the continuity equation without loss of generality we may assume that $\rho_0 =0$.
	Let $\rho \in L^\infty(I; L^2(\R^d))$ solve \eqref{CE} with $\rho|_s = 0$.
	Then
	\begin{equation*}
		\rho_+(t,x) := \begin{cases}
			0, & t \le s, \\
			\rho(t,x), & t > s
		\end{cases},
		\qquad
		\rho_-(t,x) := \begin{cases}
			\rho(t,x), & t \le s, \\
			0, & t > s.
		\end{cases}
	\end{equation*}
	also solve \eqref{CE} with $\rho_+|_s = 0 = \rho_-|_s$.
	
	By the renormalization property for any admissible function $\beta$
	the function $\beta\circ \rho_+$ solves~\eqref{CE}.
	Since $\v$ has compact support, we can write~\eqref{CE-int-by-parts}
	with a test function $\fhi$ equal to 1 in some neighborhood of the support of $\v$:
	\begin{equation*}
		\int_{\R^d} \beta(\rho_+(t,x)) \fhi(x) \, dx = \int_{\R^d} \beta(\rho_+(\tau,x)) \fhi(x) \, dx
	\end{equation*}
	for a.e. $t,\tau \in I$. Approximating the function $\beta(t)=t^2$ with a sequence of admissible functions
	and passing to the limit, we obtain
	\begin{equation*}
		\int_{\R^d} \rho_+^2(t,x) \fhi(x) \, dx = \int_{\R^d} \rho_+^2(\tau,x) \fhi(x) \, dx
	\end{equation*}
	for a.e. $t, \tau \in I$.
	Substituting $\fhi(x) = \zeta(|x|/R)$, where $\zeta$ is a smooth cutoff function
	equal to 1 on $(-1, 1)$ and equal to 0 on $\R \setminus (-2, 2)$, and passing to the limit
	as $R \to \infty$, we obtain
	\begin{equation*}
		\int_{\R^d} \rho_+^2(t,x) \, dx = \int_{\R^d} \rho_+^2(\tau, x) \, dx.
	\end{equation*}
	Since $\rho_+(\tau, \cdot) = 0$ for all $\tau \le s$, it follows that $\rho_+(t,\cdot) = 0$ for a.e. $t>s$.
	Thus $\rho(t, \cdot) = 0$ for a.e. $t>s$.
	Arguing similarly for $\rho_-$ we conclude that $\rho(t, \cdot) = 0$ for a.e. $t\le s$.
	Therefore $\rho(t,\cdot) = 0$ for a.e. $t$, hence uniqueness follows.
	
	Now let us suppose that $s \in \d I$.
	Using $\rho|_s$ as the initial or final condition, one can extend the solution $\rho$ in some open neighborhood of $\overline{I}$ using Theorem~\ref{existence-of-weak-solutions}.
	Since $\vi$ is steady, the solution can be translated forward or backward in time,
	and this reduces the problem to the previous case.
\end{proof}

Under the assumptions of previous theorem one can prove that the map $t\mapsto \rho|_t$ is not only weakly continuous, but also strongly continuous and, furthermore, the function $t\mapsto \|\rho|_t\|$ is constant (for \emph{all} values of $t$).
But in fact these two properties follow merely from the conclusion of Theorem~\ref{renormalized-solutions}:

\begin{theorem}\label{uniqueness-implies-strong-continuity}
	Let $I\subset \R$ be an open interval.
	Let $\v \in \boldsymbol{L}^2(\R^d)$ be a compactly supported divergence-free
	vector field.
	Suppose that for any $s\in I$ and any $\rho_0 \in L^2(\R^d)$ the continuity
	equation has at most one weak solution $\rho \in L^\infty(I; L^2(\R^d))$
	such that $\rho|_s = \rho_0$.
	Then for any $\rho_0 \in L^2(\R^d)$ the continuity equation has a unique solution $\widetilde{\rho} \in L^\infty(\R; L^2(\R^d)) \cap C_w(\R; L^2(\R^d))$ such that $\widetilde{\rho}|_0 = \rho_0$. Furthermore,
	\begin{enumerate}
		\item[($\ast$)] the function $t\mapsto \widetilde{\rho}(t,\cdot)$ is strongly continuous in $L^2(\R^d)$;
		\item[($\ast\ast$)] the function $t\mapsto \|\widetilde{\rho}(t,\cdot)\|$ is constant.
	\end{enumerate}
\end{theorem}

\begin{proof}
	Let $J \subset \R$ be an open interval such that $J \cap I \ne \emptyset$.
	Fix $s\in J \cap I$.
	Then by Theorem~\ref{existence-of-weak-solutions} there exists a solution $\varrho\in L^\infty(J\cup I; L^2(\R^d)) \cap C_w(\overline{J\cup I}; L^2(\R^d))$ of the continuity equation with $\varrho|_{s} = \rho|_s$ and $\|\varrho(t, \cdot)\| \le \|\rho|_s\|$ for all $t\in J\cup I$.
	By the uniqueness assumption we have $\varrho = \rho$ on~$I$.
	Since $\v$ is autonomous, the assumptions imply uniqueness of the solution on any other bounded open interval $I' \subset \R$ of the same length as $I$.
	Thus $\rho$ can be extended in a unique way to a solution $\widetilde{\rho} \in L^\infty(\R; L^2(\R^d)) \cap C_w(\R; L^2(\R^d))$, and for any $t \in \R$ we have the estimate $\|\widetilde{\rho}|_t\| \le \|\widetilde{\rho}|_s\|$.
	Interchanging $t$ and $s$ we obtain $\|\widetilde{\rho}|_s\| \le \|\widetilde{\rho}|_t\|$, hence $\|\widetilde{\rho}|_t\| = \|\widetilde{\rho}|_s\|$, i.e. $(\ast\ast)$ holds.
	It is known that if a sequence $x_n$ converges weakly to $x$ in a Hilbert space and $\|x_n\|\to \|x\|$ as $n\to \infty$, then $x_n$ converges to $x$ strongly.
	Hence ($\ast$) follows.
\end{proof}

Let us note that under the assumptions of Theorem~\ref{uniqueness-implies-strong-continuity} the weakly continuous version~$\widetilde{\rho}$ of the solution~$\rho$ becomes strongly continuous.

\begin{corollary}
	Under the assumptions of Theorem~\ref{renormalized-solutions}
	for any solution $\rho \in L^\infty(I; L^2(\R^d))$
	of~\eqref{CE} for any admissible $\beta$ and for any $t\in \overline{I}$
	we have $\beta(\rho)|_t = \beta(\rho|_t)$.
\end{corollary}

Indeed, since any admissible $\beta$ is Lipschitz, strong continuity of $t\mapsto \widetilde{\rho}(t, \cdot)$
implies that $t \mapsto \beta(\widetilde{\rho}(t, \cdot))$ is strongly continuous as well.
In other words, if $\widetilde{\rho}$ is the strongly continuous version of the solution $\rho$,
then $\beta(\widetilde{\rho})$ is the strongly continuous version of $\beta(\rho)$.
Therefore $\beta(\rho)|_{t} = \beta(\widetilde{\rho}(t, \cdot)) = \beta(\rho|_t)$.

\section{Forward and backward uniqueness implies skew-adjointness}\label{sec:uniq-skadj}

% !TeX root = skew-adjoint.tex

Theorem~\ref{uniqueness-implies-strong-continuity} implies that if the weak solutions of the Cauchy problem for the continuity equation are unique both forward and backward in time, then such solutions form a strongly continuous group $\rho|_t = S_t \rho|_0$, and such group is unique.
In this section we prove that under such assumptions the operator $A_0 = \vi \cdot \nabla$ is essentially skew-adjoint.

First we prove that for any semigroup of solutions of the continuity equation the adjoint operator $A_0^*$ is an extension of its generator (for bounded $\vi$ such result was established in \cite[Lemma~3.1]{Panov2015}, and recently it was generalized to the abstract setting in \cite[Theorem 2.1]{Panov2025}):

\begin{proposition}
	\label{adjoint-operator-extends-generator}
	Let $\v\in \boldsymbol{L}^2(\R^d)$ be a divergence-free vector field.
	Let $A$ denote the closure of the operator $A_0 \colon L^2(\R^d) \to L^2(\R^d)$
	given by $A_0(\rho) = \v \cdot \nabla \rho$
	with the domain $D(A_0) = C_c^\infty(\R^d)$.
	Let $G$ be an infinitesimal generator of $C_0$-semigroup $S_t$
	in $L^2(\R^d)$. Then the following properties are equivalent:
	\begin{enumerate}
		\item[\emph{(i)}]
		for any $\rho_0 \in L^2(\R^d)$
		the function $\rho(t,\cdot) = S_t \rho_0$ is a weak solution of
		\eqref{CE} on $(0,+\infty)$; 
		\item[\emph{(ii)}] $G \subset A^*$.
	\end{enumerate}
\end{proposition}

\begin{proof}
	\emph{$(i)\Rightarrow(ii)$.} Assume that the function $\rho(t,\cdot)=S_t\rho_0$ for each $\rho_0\in L^2(\R^d)$ is a weak solution of~\eqref{CE}. 
	If $\rho_0\in D(G)$, then $\rho\in C^1([0;+\infty), L^2(\R^d))$, and $\rho'(0)=G\rho_0$. 
	This implies that for each function $\fhi\in C^\infty_c(\R^d)$ the scalar function
		$J(t) = \int_{\R^d}\rho(t,x)\fhi(x)\,dx=(\rho(t,\cdot),\fhi)_2$
	belongs to $C^1[0;+\infty)$, and 
	\begin{equation}\label{J-prime}
		J'(0)=(\fhi, G\rho_0)_2.
	\end{equation}
	
	By continuity of $t\mapsto \rho(t, \cdot)$ the equality $\rho|_t = \rho(t,\cdot)$ holds for \emph{all} $t\in [0,+\infty)$ (see remark after Definition~\ref{def-weakly-cont-version}). Then by Proposition~\ref{weakly-cont-version} for all $t\in [0,+\infty)$ and for all $\fhi \in C^\infty_c(\R^d)$ we have
	\begin{equation*}
		J(t) - \int_{\R^d} \rho_0(x) \fhi(x) \, dx
		= \int_0^t (A_0 \fhi, \rho(t,\cdot))_2 \, dt.
	\end{equation*}
	Hence (by continuity of $t\mapsto (A_0 \fhi, \rho(t,\cdot))_2$)
	\begin{equation*}
		J'(0) = (A_0 \fhi, \rho(0,\cdot))_2 = (A_0 \fhi, \rho_0)_2.
	\end{equation*}
	Then by \eqref{J-prime}
	\begin{equation*}
		(A_0\fhi, \rho_0)_2 = (\fhi, G\rho_0)_2.
	\end{equation*}
	Since $\fhi\in C^\infty_c(\R^d)$ was arbitrary, by the definition of adjoint operator it follows that $\rho_0\in D(A^*)$ and $A^*\rho_0=G\rho_0$. 
	Hence $G\subset A^*$.
	
	\emph{$(ii)\Rightarrow (i)$.}
	Assume that $G\subset A^*$ and $\rho_0\in D(G)$. Then $\rho(t,\cdot)=S_t\rho_0$ is a $C^1$-function with values in $L^2(\R^d)$ such that $\d_t \rho = GS_t\rho_0 = G\rho(t, \cdot)$. This implies that for arbitrary $\fhi\in C^\infty_c(\R^d)$ and $t\geq0$
	\begin{equation*}
		\d_t (\rho(t,\cdot), \fhi)_2=(G\rho(t,\cdot), \fhi)_2=(A^*\rho(t,\cdot), \fhi)_2=(\rho(t,\cdot), A_0\fhi)_2,
	\end{equation*}
	where $A_0\fhi=\v \cdot \nabla \fhi$, that is,
	\begin{equation*}
		\d_t \int_{\R^d}\rho(t, x)\fhi(x)\,dx-\int_{\R^d}\rho(t,x)\v(x) \cdot \nabla \fhi(x)\,dx=0.
	\end{equation*}
	Of course, the equality above holds also for any $\fhi \in C^1_c(\R^d)$, since any such function can be approximated by functions from $C^\infty_c(\R^d)$.
	Hence, by Proposition~\ref{equivalent-def-of-weak-sol} $\rho$ is a weak solution of \eqref{CE}. 
	
	Now, let $\rho_0\in L^2(\R^d)$ and $\rho(t,\cdot)=S_t\rho_0$. Since the generator of a $C_0$-semigroup is dense, we can find a sequence $\{\rho_{0,n}\}\subset D(G)$ converging to $\rho_0$ as $n\to\infty$ in $L^2(\R^d)$. Then by just proven above $\rho_n(t,\cdot)=S_t\rho_{0,n}$ is a weak solution of \eqref{CE}.
	Hence $\rho_n$ satisfies \eqref{CE-D'}, i.e.
	\begin{equation}\label{rho-n-CE-weak-eq}
		\int_0^{+\infty} \int_{\R^d} \rho_n(t,x) \d_t \Phi(t,x) \, dx \,dt
		+ \int_0^{+\infty} \int_{\R^d} \rho_n \v \cdot \nabla_x \Phi(t,x) \, dx \,dt = 0.
	\end{equation}
	By the well-known property of semigroups the map $t\mapsto \|S_t\|$ is locally bounded, hence
	\begin{equation*}
		\|\rho_n(t,\cdot)-\rho(t,\cdot)\|_2\leq \|S_t\|\cdot\|\rho_{0,n}-\rho_0\|_2\to0
	\end{equation*}
	uniformly in $t$ on any segment $[0;T]$. In particular, $\rho_n\to \rho$ as $n\to\infty$ in
	$L^2_{loc}((0;+\infty)\times \R^d)$.
	Passing to the limit as $n\to\infty$ in \eqref{rho-n-CE-weak-eq} we conclude that $\rho$ is a weak solution of \eqref{CE}. 
\end{proof}

\begin{proposition}
	\label{uniqueness-implies-skew-adjointness}
	Let $\v \in \boldsymbol{L}^2(\R^d)$ be a compactly supported divergence-free
	vector field.
	Let $I \subset \R$ be a nonempty bounded open interval and suppose that
	for any $s\in I$ and any $\rho_0 \in L^2(\R^d)$ the continuity
	equation has at most one weak solution $\rho \in L^\infty(I; L^2(\R^d))$
	such that $\rho|_s = \rho_0$.
	Then the operator $A_0 \colon L^2(\R^d) \to L^2(\R^d)$ given by $A_0 = \v \cdot
	\nabla$ with the domain $D(A_0) = C_c^\infty(\R^d)$ is essentially skew-adjoint.
\end{proposition}

\begin{proof}
	By Theorem~\ref{uniqueness-implies-strong-continuity} 
	the weak solutions $\rho$ of the initial value problem
	for~\eqref{CE} on $\R$ with the initial condition $\rho|_0 = \rho_0$, $\rho_0 \in L^2(\R^d)$, form a $c_0$-unitary group $S_t$ given by $S_t \rho_0 = \rho|_t$.
	Let $G$ denote the generator of this group.
	By Stone's theorem $G^* = -G$.
	
	Let $A$ denote the closure of $A_0$.
	By Proposition~\ref{adjoint-operator-extends-generator}
	we have $G \subset A^*$, hence $A^{**} \subset G^* = -G$.
	Since $A = A^{**}$, it follows that $A \subset -G$,
	i.e. $-G$ is a skew-adjoint extension of $A$.
	
	By the known theory of deficiency indices (see e.g. \cite[Theorem 4.2 and Corollary 3.12]{Arendt2023}\footnote{In the framework of self-adjoint extensions of closed symmetric operators one could also refer to \cite[Corollary after Theorem X.2]{SimonReedV2}. In the framework of skew-adjoint extensions of closed skew-symmetric operators the statement of such result (and its proof) is completely analogous.})
	there is a one-to-one correspondence between the skew-adjoint extensions of the operator $A$ and the isometric isomorphisms between the spaces $\mathscr{K}_+ := \operatorname{Ker}(A^*+E)$ and $\mathscr{K}_- := \operatorname{Ker}(A^*-E)$.
	Let $U$ denote such isometric isomorphism corresponding to $-G$.
	Let $-\tilde G$ denote the skew-adjoint extension of $A$ corresponding to $-U$.
	
	By Stone's theorem $\tilde G$ generates a
	unitary group $\tilde S_t$. %different from $S_t$.
	By Proposition~\ref{adjoint-operator-extends-generator} the group $\tilde S_t$ also solves the continuity equation, i.e. for any $\rho_0 \in L^2(\R^d)$ the function $\tilde \rho(t) := \tilde S_t \rho_0$ solves \eqref{CE} and satisfies $\tilde \rho|_0 = \rho_0$.
	By the uniqueness of solutions to the Cauchy problem for \eqref{CE} we conclude that the semigroups $S_t$ and $\tilde S_t$ coincide.
	Hence their generators, $G$ and $\tilde G$, are equal.
	This implies that $U=-U$, hence $\mathscr{K}_+=\mathscr{K}_-=\{0\}$.
	Therefore $-A = A^*$.
\end{proof}

\section{Skew-adjointness implies chain rule and renormalization}\label{sec:skadj-renorm}

% !TeX root = skew-adjoint.tex

\subsection{Chain rule}
In the smooth setting a simple computation shows that if $\dive(\rho \vi) = f$ then for any $\beta \in C^1(\R)$ we have $\dive (\beta(\rho)\vi) = \beta'(\rho) f$.
We are going to show that such implication is preserved even in the weak setting, if the operator $\vi \cdot \nabla$ is essentially skew adjoint.

\begin{proposition}
	\label{chain-rule-for-skew-adjoint}
	Let $\v\in\boldsymbol{L}^2_{loc}(\R^d)$ be a divergence-free vector field
	such that the corresponding operator $A_0=\v\cdot \nabla$ acting on $L^2(\R^d)$
	with domain $D(A_0)=C^\infty_c(\R^d)$ is essentially skew-adjoint (i.e. $A^*_0=-\overline{A_0}$).
	Let $\beta$ be an admissible function.
	If $\rho\in L^2(\R^d)$ and $f\in L^2(\R^d)$ satisfy
	\begin{equation*}
		\dive (\rho \v)=f \qquad \text{in } \ \mathscr{D}'(\R^d)
	\end{equation*}
	then
	\begin{equation*}
		\dive (\beta(\rho)\v)=\beta'(\rho) f \qquad \text{in }\
		\mathscr{D}'(\R^d).
	\end{equation*}
\end{proposition}

\begin{proof}
	By the definition of an adjoint operator
	\begin{equation*}
		A_0^*u=-\dive(u\v)\quad\forall u\in D(A_0^*)=\{u\in L^2(\R^d) \;:\; \dive(u\v)\in L^2(\R^d)\}, 
	\end{equation*}
	where the divergence operator is understood in the sense of distributions.
	Then $\rho\in D(A^*_0)$ and $A^*_0 \rho=-f$. Since $A_0$ is essentially
	skew-adjoint, it follows that $\rho\in D(\overline{A_0})=D(A^*_0)$ and
	$\overline{A_0}\rho=f$.
	Then by the definition of closure of an operator, there exists a sequence
	$\{\rho_n\}\subset D(A_0)=C^\infty_c(\R^d)$ such that
	\begin{equation}
		\label{graph_convergence}
		\left\{
		\begin{aligned}
			& \rho_n\to \rho\\
			& f_n := A_0 \rho_n=\v\cdot\nabla \rho_n\to f
		\end{aligned}
		\right.
	\end{equation}
	in $L^2(\R^d)$ as $n\to\infty$.
	Since $\rho_n$ converges to $\rho$ strongly, we can extract a subsequence
	(without relabelling) such that $\rho_n$ converges to $\rho$ a.e.
	
	Since $A_0 \rho_n = f_n$ and $\rho_n$ is smooth, we have
	\begin{equation*}
		\dive(\beta(\rho_n) \v) = \v \nabla \beta(\rho_n) = \beta'(\rho_n) \v \nabla
		\rho_n = \beta'(\rho_n) \dive(\rho_n \v) = \beta'(\rho_n) f_n
	\end{equation*}
	(in the sense of distributions).
	Hence
	\begin{equation}
		\label{approx-div-rho_n}
		\dive(\beta(\rho_n) \v) = (\beta'(\rho_n)) (f_n-f) + \beta'(\rho_n) f.
	\end{equation}
	
	By Remark~\ref{renorm_is_L2} $\beta\circ \rho\in L^2(\R^d)$. Then, by the mean
	value theorem and \eqref{graph_convergence}
	\begin{equation*}
		\|\beta\circ \rho_n-\beta\circ
		\rho\|\le\|\beta'\|_{\infty}\|\rho_n-\rho\|\to0,
	\end{equation*}
	hence the left-hand side of \eqref{approx-div-rho_n}
	converges to $\dive (\beta(\rho) \v)$ in the sense of distributions.
	Since $\beta'$ is bounded, we have
	\begin{equation*}
		\|(\beta'(\rho_n)) (f_n-f)\| \le \|\beta'\|_\infty \|f_n - f\| \to 0
	\end{equation*}
	as $n\to \infty$.
	On the other hand, by continuity of $\beta'$ we have $\beta'(\rho_n)\to \beta'(\rho)$ a.e.
	Since $\beta'$ is bounded, then by dominated convergence $\beta'(\rho_n)f\to\beta'(\rho)f$ in $L^2(\R^d)$.
	Thus, passing to the limit in \eqref{approx-div-rho_n},
	we get $\dive(\beta(\rho) \v) = \beta'(\rho) f$.
\end{proof}

\begin{proposition}\label{space-time-skew-adj-and-renorm}
	Let $\v \in L^2(\R^d)$ be a compactly supported divergence-free vector field.
	Suppose that the operator $B_0(\rho) := \d_t \rho + \v \cdot \nabla_x \rho$
	with $D(B_0) = C^\infty_c(\R^{d+1})$ is essentially skew-adjoint in $L^2(\R^{d+1})$.
	Then for any open interval $I\subset \R$ the vector field $\v$
	has the renormalization property in the class $L^2(I \times \R^d)$.
\end{proposition}

\begin{proof}
	Suppose that $\rho \in L^2(I \times \R^d)$ solves \eqref{CE}.
	Let $\psi \in C^\infty_c(I)$, and let us extend $\rho$ and $\psi$
	with zeroes to $\R^{d+1}\setminus (I \times \R^d)$ and to $\R \setminus I$
	respectively.
	Then the function $\psi \rho$
	belongs to $L^2(\R^{d+1})$ and solves
	\begin{equation}
		\label{time-cutoff-CE}
		\d_t (\psi \rho) + \dive(\psi \rho \v) = \rho \d_t \psi
	\end{equation}
	in $\mathscr{D}'(\R^{d+1})$.
	
	The right-hand side of \eqref{time-cutoff-CE} belongs to
	$L^2(\R^{d+1})$. Hence by Proposition~\ref{chain-rule-for-skew-adjoint}
	for any admissible function $\beta$ we have
	\begin{equation}
		\d_t (\beta(\psi \rho)) + \dive(\beta(\psi \rho) \v) = \beta'(\psi \rho)\rho
		\d_t \psi
	\end{equation}
	For any open interval $J$ with compact closure contained in $I$ there exists a function $\psi\in C^\infty_c(I)$ such that $\psi=1$ on $J$. Then it follows from the above relation that \eqref{renorm} holds in $\mathscr{D}'(J\times\R^d)$.
	By arbitrariness of $J$ we conclude that
	\eqref{renorm}
	holds in $\mathscr{D}'(I\times \R^d)$.
\end{proof}

\subsection{Essential skew-adjointness of the extended operator and renormalization}

From the previous section (see Proposition~\ref{space-time-skew-adj-and-renorm}) we know that if the operator $\d_t + \vi \cdot \nabla$ is essentially skew adjoint, then the renormalization property holds.
In this section we prove that if $\vi \cdot \nabla$ is essentially skew adjoint,
then the <<extended operator>> $\d_t + \vi \cdot \nabla$ is essentially skew adjoint as well.
In fact, we prove this result in a slightly more general setting, when $\vi \cdot \nabla$ is replaced by an abstract essentially skew-adjoint operator.

Let~$H$ be a Hilbert space and let $A_0$ be a densely defined skew-symmetric linear operator on $H$ with the domain $D(A_0)$. Let $A$ be the closure of $A_0$ and let $X$ and $X_0$ denote the spaces $D(A)$ and $D(A_0)$ equipped with the graph norm $\|x\|_H + \|Ax\|_H$ respectively.
Since the operator $A$ is the closure of $A_0$, $X$ is a Banach space and $X_0$ is its dense subspace.
Let $\widetilde{X}:= C^1_c(\R; X)$ and $\widetilde{X}_0:= C^1_c(\R; X_0)$. If 
$u\in \widetilde{X}_0$
then $u'\in C_c(\R; X_0)$ and $Au=\bigl(s\mapsto A(u(s))\bigr)$ belongs to $C^1_c(\R; H)$ and satisfies $(Au)'(s)=A(u'(s))$ in $H$. On the space $\widetilde{H}:=L^2(\R; H)$ we consider the 
densely defined % (by Lemma \ref{dense-subspace})  
operator $\widetilde{A}_0$ given by
\begin{equation}\label{A_0-tilde-def}
	(\widetilde{A}_0 u)(s) := u'(s) + A_0(u(s)), \qquad s\in \R, \qquad u\in D(\widetilde{A}_0) := \widetilde{X}_0.
\end{equation}

In general a sum of two essentially skew-adjoint operators with common dense domain may fail to be essentially skew-adjoint.
For instance, in the case $H=L^2(0,1)$
the operators $Au = u'-u'''$ and $Bu = u'''$ with $D(A)=D(B)=\{u\in H^3(0,1) : u(0)=u(1)=0, \ u'(0)=u'(1)\}$ are skew-adjoint, while their sum is not, since $D((A+B)^*) = H^1(0,1) \ne H_0^1(0,1) = D(\overline{A+B})$
(for self-adjoint operators a similar construction is given in \cite{Renardy2011}).
Another example of such pair of operators, for which the failure of skew-adjointness of their sum is not related to the boundary conditions, is discussed in Remark~\ref{sum-of-skadj-is-not}.
Nevertheless, in the particular case considered in \eqref{A_0-tilde-def} the following result holds:

\begin{proposition}\label{skew-adjoint-extra-dimension}
	If the operator $A_0$ is essentially skew-adjoint,
	then the operator $\widetilde{A}_0$ is essentially skew-adjoint.
\end{proposition}

\begin{proof}
	Since $A$ is skew-adjoint, by Stone's theorem the operator $-A$ generates a unique unitary $c_0$-group of isomorphisms on $H$, which we denote by $T_t$, where $t\in \R$.
	Consider the following unitary $c_0$-group on $\widetilde{H}$:
	\begin{equation}
		(G_t u)(s) := T_t(u(s-t)).
	\end{equation}

	Let $-B$ denote the generator of $G_t$ which is skew-adjoint by Stone's theorem. We are going to prove that $B$ is the closure of $\widetilde{A}_0$.
	
	\emph{Step 1.}
	Let 
	\begin{equation*}
		F := W^{1,2}(\R; H) \cap L^2(\R; X).
	\end{equation*}	
	(Here $W^{1,2}(\R; H)$ denotes the Sobolev space, which consists of all $u\in L^2(\R; H)$, for which there exists a weak derivative $u'\in L^2(\R; H)$, see e.g. \cite[Definition 2.2.22]{GP2006}.)
	Let us show that $F \subset D(B)$ and 
	\begin{equation}\label{B-action}
		(Bu)(s) = u'(s) + A(u(s)), \quad s\in \R, \quad u\in F.
	\end{equation}
	
	Suppose that $u\in F$.
	Then
	\begin{equation*}
		\frac{(G_t u)(s) - u(s)}{t} = T_t \frac{u(s-t) - u(s)}{t} + \frac{T_t (u(s)) - u(s)}{t}.
	\end{equation*}
	Since
	\begin{equation*}
		\lim_{t\to 0} \frac{u(s-t) - u(s)}{t} = -u'(s)\qquad\text{in}\quad \widetilde{H}
	\end{equation*}
	and
	\begin{equation*}
		\lim_{t\to 0} \frac{T_t (u(s)) - u(s)}{t}= -A (u(s))\qquad\text{in}\quad \widetilde{H},
	\end{equation*}
	we find that there exists
	\begin{equation*}
		-Bu = \lim_{t\to 0} \frac{G_t u - u}{t}\in \widetilde{H}
	\end{equation*}
	and
	\begin{equation*}
		(Bu)(s) = u'(s) + A(u(s)), \quad s\in \R.
	\end{equation*}
	Thus $u \in D(B)$. 
	
	\emph{Step 2.}
	Clearly, $\widetilde{X}_0=C^1_c(\R; X_0)\subset F$. 
	By Step 1 $F\subset D(B)$ and \eqref{B-action} holds which coincides with \eqref{A_0-tilde-def} for all $u\in\widetilde{X}_0$.
	It follows that $\widetilde{A}_0\subset B$, and so $B$ is a closed extension of $\widetilde{A}_0$. 
	This implies that $\widetilde{A}_0$ is closable. 
	Let us denote $\widetilde{A}$ the closure of $\widetilde{A}_0$, then $\widetilde{A} \subset B$.
	It remains to prove the inclusion $B \subset \widetilde{A}$.
	
	\emph{Step 3.}
	Let us show that $F \subset D(\widetilde{A})$.
	First we will prove that $C^1_c(\R; X) \subset D(\widetilde{A})$ and for any~$u\in C^1_c(\R; X)$
	\begin{equation}\label{widetilde-A}
		(\widetilde{A} u)(s) = u'(s) + A(u(s)).
	\end{equation}
	By Lemma~\ref{dense-subspace} there exist and open interval $(a,b)$ and a sequence $v_n \in C^1_c(\R; X_0)$
	such that for all $n\in \N$ we have $\supp v_n \subset (a,b)$ and $\|v_n - u\|_{C^1(\R; X)} \to 0$ as $n\to \infty$. 
	Since $v_n \in D(\widetilde{A}_0)$ and $\widetilde{A}$ is the closure of $\widetilde{A}_0$, we have
	\begin{equation*}
		(\widetilde{A} v_n)(s) = (\widetilde{A}_0 v_n)(s) = v_n'(s) + A_0 (v_n(s)).
	\end{equation*}
	Since $A$ is the closure of $A_0$ and $v_n \to u$ in $C^1_c(\R; X)$, for all $s\in \R$ we have
	$
		(\widetilde{A} v_n)(s) \to (\widetilde{A} u)(s).
	$
	Furthermore, $\widetilde{A} v_n \to \widetilde{A} u$ in $C_c(\R; H)$ (and consequently in $\widetilde{H}$, since $\supp v_n \subset (a,b)$). 
	At the same time $v_n\to u$ in $C_c(\R; H)$ (hence in $\widetilde{H}$).
	Since $\widetilde{A}$ is closed, it follows that $u\in D(\widetilde{A})$ and~\eqref{widetilde-A} holds.
	
	Now let us show that any compactly supported $u \in F$ belongs to $D(\widetilde{A})$.
	For any $\eps>0$ let $(u)_\eps := u * \omega_\eps$. 
	Clearly $(u)_\eps$ belongs to $C^1_c(\R; X)$ and at the same time $(u)_{\eps}\to u$ in $L^2(\R; X)$ and in $\widetilde{H}$, while $(u)_\eps' = (u')_\eps \to u'$ in $\widetilde{H}$ as $\eps\to 0$. 
	Then again $u\in D(\widetilde{A})$ and \eqref{widetilde-A} holds by the closedness of $\widetilde{A}$.
	
	Finally, consider arbitrary $u \in F$. 
	Let $w$ be a smooth function such that $w(s)=1$ for $|s|<1$ and $w(s)=0$ for $|s|>2$.
	For any $R>0$ let $u_R(s) = u(s)w(s/R)$.
	Then $u_R \to u$ in $L^2(\R; X)$ and in $\widetilde{H}$, while $u_R' \to u'$ in $\widetilde{H}$ as $R\to \infty$.
	Hence again $u\in D(\widetilde{A})$ and \eqref{widetilde-A} holds by the closedness of $\widetilde{A}$.
	
	\emph{Step 4.}
	Let $E$ denote the identity operator. Since $B$ is a skew-adjoint operator, the operator $(E+B)$ has bounded inverse $(E+B)^{-1}\colon\widetilde{H}\to D(B)$. Let $u := (E+B)^{-1} f$ where $f\in \widetilde{H}$.
	By the known representation of the resolvent of the generator of $c_0$-semigroup, we find that 
	\begin{equation}\label{resolvent}
		u = \int_0^{+\infty} e^{-t} G_t f \, dt,
		\quad \text{i.e.} \quad
		u(s)= \int_0^{+\infty} e^{-t} T_t f(s-t) \, dt.
	\end{equation}
	Let us show that if $f\in C_c(\R; X)$ then $u\in F$ and 
	\begin{equation}\label{u3}
		\|u\|_{\widetilde H}\leqslant\|f\|_{\widetilde H}.
	\end{equation}
	
	Firstly, let us show that $u\in C(\R; X)$. Since the space $X$ is invariant for the unitary group $T_t$ and $f(s-t)\in X$ for all $s,t\in\R$, the function $g_s(t):=e^{-t}T_tf(s-t)$ takes values in $X$ for all $s,t\in\R$.
	Moreover, since 	 
	\begin{equation}\label{X-norm}
		\|T_tv\|_X=\|T_tv\|_H+\|AT_tv\|_H=\|v\|_H+\|T_tAv\|_H=\|v\|_H+\|Av\|_H=\|v\|_X
	\end{equation}
	for any $v\in X$ and $t\in\R$, and
	\begin{equation*}
		\|T_tv-v\|_X=\|T_tv-v\|_H+\|T_tAv-Av\|_H\to0
	\end{equation*}
	as $t\to0$ for any $v\in X$, the inclusion $f\in C_c(\R; X)$ implies that $g_s\in C_c(\R;X)$ for each $s\in\R$, and so $u(s)\in X$ as the integral \eqref{resolvent} of a continuous function $g_s$ with values in Banach space $X$. In particular,
	\begin{equation*}
		Au(s)=\int_0^{+\infty} e^{-t} AT_t f(s-t) \, dt,\qquad s\in\R.
	\end{equation*}
	Continuity of $u$ follows from continuity of $g_s(t)$ with respect to $s$ and uniform convergence of the integral \eqref{resolvent} with respect to the graph norm.
	
	Secondly, we'll show that $u\in L^2(\R;X)$, $Au\in \widetilde{H}$ and \eqref{u3}. Using \eqref{X-norm} we get
	\begin{equation*}
		\|u(s)\|_{X} 
		\le \int_0^{+\infty} e^{-t} \|T_t f(s-t)\|_{X} \, dt 
		=\int_0^{+\infty} e^{-t} \|f(s-t)\|_{X} \, dt
		= (\gamma * \fhi) (s),
	\end{equation*}
	where $\gamma(t) = \theta(t) e^{-t}$, $\theta$ is the Heavyside function
	and $\fhi(t) = \|f(t)\|_{X}$. Then by the Young's convolution inequality
	\begin{equation}\label{u2}
		\|u\|_{L^2(\R; X)} \le \|\gamma\|_1 \|\fhi\|_2 = \|\fhi\|_2 = \|f\|_{L^2(\R; X)}<+\infty.
	\end{equation}
	It follows from \eqref{u2} that $u\in L^2(\R;X)$ and $Au\in L^2(\R; H)$ (since $\|Au(s)\|_H\le\|u(s)\|_X$). The inequality \eqref{u3} is proven in the same way as \eqref{u2}.
	
	Lastly, let us show that $u\in W^{1,2}(\R; H)$. Indeed, using the representation \eqref{resolvent} we get 
	\begin{align}
		u'(s) &= \frac{d}{ds} \int_0^{+\infty} e^{-t} T_t f(s-t) \, dt \\
		&= \frac{d}{ds} \int_{-\infty}^{s} e^{t-s} T_{s-t} f(t) \, dt \\
		&= f(s) - \int_{-\infty}^{s} e^{t-s} T_{s-t} f(t) \, dt - \int_{-\infty}^{s} e^{t-s} AT_{s-t} f(t) \, dt \\
		&= f(s) - u(s) - A u(s) \in L^2(\R; H).
	\end{align}
	
	\emph{Step 5.}
	Let $u \in D(B)$, then $f:= u + Bu$ belongs to $\widetilde{H}$.
	We can find a sequence $f_k \in C_c(\R; X)$ which approximates $f$ in $\widetilde{H}$:
	$\|f_k - f\|_{\widetilde{H}} \to 0$ as $k\to \infty$.
	Let $u_k := (E + B)^{-1} f_k$. 
	By the previous step we have $u_k \in F$, and by \eqref{u3} we have $\|u_k - u\|_{\widetilde{H}} \to 0$ as $k\to \infty$.
	
	By Steps 1 and 3 we have $F \subset D(B) \cap D(\widetilde{A})$.
	Then for each $k$ we have $\widetilde{A} u_k = B u_k = f_k - u_k$.
	Since $u_k \to u$ and $f_k - u_k \to f-u$ in $\widetilde{H}$,
	it follows that $u \in D(\widetilde{A})$ and $\widetilde{A}(u) = f-u$.
	Thus $D(B) \subset D(\widetilde{A})$, i.e. $B \subset \widetilde{A}$.
\end{proof}

\begin{remark}
	In \cite[Lemma 5.1]{Panov2015} Proposition~\ref{skew-adjoint-extra-dimension} was proved for the operator~$A_0$ given by~\eqref{differentiation-along-v} with bounded divergence-free $\v$.
\end{remark}

The implication (i)$\to$(ii) of Theorem~\ref{main}
is a direct consequence of the following result:

\begin{proposition}\label{spacial-skew-adj-and-renorm}
	Let $\v \in L^2(\R^d)$ be a compactly supported divergence-free vector field.
	Suppose that the operator $A_0 \rho := \v \cdot \nabla \rho$
	with $D(A_0) = C^\infty_c(\R^{d})$ is essentially skew-adjoint in $L^2(\R^{d})$.
	Then for any open interval $I\subset \R$ the vector field $\v$
	has the renormalization property in the class $L^2(I \times \R^d)$.
\end{proposition}

\begin{proof}
	Let $\widetilde{A}_0$ denote operator defined in \eqref{A_0-tilde-def}.
	(Recall that $D(\widetilde{A}_0) = C^1_c(\R; X_0)$, where $X_0$ denotes $C^\infty_c(\R^d)$ endowed with the graph norm of $A_0$.)
	By Proposition~\ref{skew-adjoint-extra-dimension} the operator $\widetilde{A}_0$ is essentially skew adjoint.
	
	Now let us consider the operator $B_0(\rho) := \d_t \rho + \v \cdot \nabla_x \rho$ with the domain $D(B_0) = C^\infty_c(\R^{d+1}) \subset L^2(\R^{d+1})$.
	By Lemma~\ref{dense-subspace} the inclusion $C^\infty_c(\R^{d+1}) \subset C^1_c(\R; X_0)$ holds and is dense with respect to the graph norm of $\widetilde{A}_0$.
	Then by Proposition~\ref{space-time-skew-adj-and-renorm} $\vi$ has the renormalization property.
\end{proof}

\begin{corollary}\label{strongly-cont-version-is-given-by-group}
	If $\vi \in \boldsymbol{L}^2(\R^d)$ is a compactly supported divergence-free vector field
	and the operator $A_0 \rho := \v \cdot \nabla \rho$
	with $D(A_0) = C^\infty_c(\R^{d})$ is essentially skew-adjoint in $L^2(\R^{d})$,
	then for any $T>0$ and $\rho_0 \in L^2(\R^d)$ there exists a unique solution $\rho \in L^\infty(0,T; L^2(\R^d))$ of \eqref{CE} with $\rho|_0 = \rho_0$.
	Furthermore, the strongly continuous version $\widetilde{\rho}$ of $\rho$ is given by the group $S_t$ generated by $-A$, where $A$ is the closure of $A_0$, i.e. for all $t\in [0,T]$ we have
	$\widetilde{\rho}(t,\cdot) = S_t \rho_0 = \rho|_t$.
\end{corollary}

\begin{proof}
	By Proposition~\ref{adjoint-operator-extends-generator} with $G=-A$ 
	the function $\eta := S_t \theta$ solves \eqref{CE} and satisfies $\eta|_0 = \rho_0$.
	On the other hand, by Proposition~\ref{spacial-skew-adj-and-renorm} and Theorem~\ref{renormalized-solutions} we have uniqueness of weak solutions of the initial value problem for \eqref{CE}, hence $\eta(t, \cdot) = \rho(t, \cdot)$ for a.e. $t$.
	Then, since $\rho(t, \cdot) = \widetilde{\rho}(t,\cdot)$ for a.e. $t$ and $t\mapsto \eta(t, \cdot)$ is strongly continuous, it follows that $\eta(t, \cdot) = \widetilde{\rho}(t, \cdot)$ for \emph{all} $t$.
\end{proof}

\section{Unequivalence of forward and backward uniqueness}\label{unequivalence}

% !TeX root = skew-adjoint.tex

In view of Theorem~\ref{renormalized-solutions}
the renormalization property of $\v$ implies uniqueness
of weak solutions of the continuity both forward and backward in time.
In this section we prove that in the autonomous two-dimensional setting
uniqueness forward in time is equivalent to uniqueness backward in time,
at least for the class of all compactly supported bounded vector fields.

The main feature of the two-dimensional setting is that if the vector field $\vi$ is bounded and divergence-free, then it can be written as the skew gradient of some Lipschitz function (called \emph{stream function} of $\vi$): $\vi = \nabla^\perp f$, where $\nabla^\perp f = (-\d_2 f, \d_1 f)$.
Alberti, Bianchini and Crippa introduced the following property:

\begin{definition}
	Let $f\colon \R^2 \to \R$ be a continuous function.
	Let $Z\subset \R^2$ denote the critical set of~$f$
	(i.e. the set of all points where $f$ is not differentiable, or has $\nabla f = 0$).
	Let $E^*$ be the union of all connected components of the level sets
	of $f$ which have strictly positive length (i.e. Hausdorff measure $\H^1$).
	We say that $f$ has the \emph{weak Sard property}, if
	\begin{equation*}
		f_\# (\L^2 \rest (Z\cap E^*)) \perp \L^1. \label{WSP}
	\end{equation*}
\end{definition}

Here $\mu \perp \nu$ means that $\mu$ and $\nu$ are mutually singular measures, and $\mu \rest A$ denotes the restriction of the measure $\mu$ to the set $A$.
Note that the set $E^*$ is Borel, since $f$ is continuous
(see \cite[Proposition 6.1]{ABC_2013_LipSard}).

Let $T>0$ and $I=(0,T)$.
Let $\Lip_c(\R^d)$ denote the set of all compactly supported Lipschitz functions on $\R^d$.
The following elegant geometric uniqueness criteria was proved in \cite[Theorem 4.7]{ABC_2014}:
\begin{theorem}\label{uniq-Linf} 
	Let $f\in \Lip_c(\R^2)$ and $\v = \nabla^\perp f$.
	Then the following two properties are equivalent:
	\begin{enumerate}
		\item
		If~$\rho \in L^\infty(I; L^\infty(\R^2))$
		solves \eqref{CE} with the initial condition $\rho|_0 = 0$,
		then $\rho = 0$ a.e.
		\item
		The function $f$ has the weak Sard property.
	\end{enumerate}
\end{theorem}

In fact the weak Sard property is sufficient for uniqueness of weak solutions
in a slightly different class (see e.g. \cite[Remark A.5]{GK24}):

\begin{theorem}\label{uniq-L1}
	Suppose that $f\in \Lip_c(\R^2)$ has the weak Sard property and $\v = \nabla^\perp f$. 
	Let~$p\in[1,+\infty)$.
	If~$\rho \in L^\infty(I; L^p(\R^2))$
	solves \eqref{CE} with the initial condition $\rho|_0 = 0$,
	then $\rho = 0$ a.e.
\end{theorem}

Now we are in a position to collect some properties which are equivalent for autonomous planar divergence-free vector fields:

\begin{theorem}\label{uniqueness-forward-and-backward}
	Suppose that $\v \in \boldsymbol{L}^\infty(\R^2)$ is
	compactly supported and divergence-free.
	Then the following properties are equivalent:
	\begin{enumerate}
		\item[(i)] for any $\rho_0 \in L^2(\R^2)$
		there exists a unique solution $\rho\in L^\infty(I; L^2(\R^2))$
		of \eqref{CE} with the initial condition $\rho|_0 = \rho_0$;
		\item[(ii)] for any $\rho_0 \in L^2(\R^2)$
		there exists a unique solution $\rho\in L^\infty(I; L^2(\R^2))$
		of \eqref{CE} with the final condition $\rho|_T = \rho_0$.
		\item[(iii)] the stream function of $\vi$ has the weak Sard property;
		\item[(iv)] the operator $\vi \cdot \nabla$ is essentially skew adjoint.
	\end{enumerate}
\end{theorem}

\begin{proof}[Proof of Theorem~\ref{uniqueness-forward-and-backward}]
	By changing the variables $\widetilde\rho(t,x) :=\rho(T-t, x)$
	it is easy to see that $\rho$ solves \eqref{CE}
	if and only if $\widetilde \rho$ solves $\d_t \widetilde\rho - \dive (\widetilde \rho \v) = 0$. Since the class of bounded divergence-free
	vector fields is invariant under the map $\v \mapsto -\v$,
	then the implications (i)$\Rightarrow$(ii) and (ii)$\Rightarrow$(i) are equivalent. Therefore, to establish equivalence of statements (i), (ii), (iii), it is sufficient to prove that (i) implies (iii) and (iii) implies (ii).
	Furthermore, by linearity of \eqref{CE} it is sufficient
	to consider the case $\rho_0 = 0$.
	
	Let $f\colon \R^2 \to \R$ be a Lipschitz function such that $\v = \nabla^\perp f$.
	Existence of such function follows from the equality $\dive \v =0$.
	Furthermore, since $\v$ is compactly supported, the function $f$
	can be chosen to have compact support as well.
	
	If the support of $\v$ is contained in the ball $\overline{B}_R(0)$
	for some $R>0$ and $\rho\in L^\infty(I; L^2(\R^d))$ solves
	then $\rho=0$ a.e. in the complement of $\overline{B}_R(0)$.
	This shows that any solution of \eqref{CE} with zero initial
	condition has compact support (in space, uniformly in time).
	Then assumption~(i) implies that we have uniqueness (forward in time) in the class $\rho \in L^\infty(I; L^\infty(\R^2))$.
	
	Now, assume that statement (iii) holds.
	By Theorem~\ref{uniq-Linf} this implies that $f$ has the weak Sard property. 
	Since $-\v = -\nabla^\perp f$ and clearly $-f$ has the weak Sard property,
	it immediately follows from Theorem~\ref{uniq-L1} that (ii) holds.
	
	Ultimately, by Theorem~\ref{main} property (iv) is equivalent to conjunction of (i) and (ii), which is equivalent to (iii).
\end{proof}

In the case $d=2$ under the assumptions of Theorem~\ref{uniqueness-forward-and-backward} 
weak solutions of the Cauchy problem for the continuity equation
are unique forward in time if and only if they are unique backward in time.
In the case $d\ge 3$ uniqueness forward in time is not equivalent
to uniqueness backward in time, and each of these properties
is strictly weaker than the renormalization property:

\begin{theorem}%[Fusilli flow]
	\label{fusilli}
	If $d\ge 3$, then
	there exists a %compactly supported 
	divergence-free vector field
	$\v \in \boldsymbol{L}^\infty(\R^d)$ with the following properties:
	\begin{enumerate}
		\item[(i)] for any $\rho_0 \in L^2(\R^d)$
		there exists a unique solution $\rho\in L^\infty(I; L^2(\R^d))$
		of \eqref{CE} (forward in time) with the initial condition $\rho|_0 = \rho_0$;
		\item[(ii)] $\v$ does not have the renormalization property
		in the class $L^\infty(I; L^\infty(\R^d))$
		(and consequently in the class $L^\infty(I; L^2(\R^d))$).
		Furthermore, there exists a nontrivial solution $\rho\in L^\infty(I; L^2(\R^d))$
		of \eqref{CE} (backward in time) with the final condition $\rho|_T = 0$.
	\end{enumerate}
\end{theorem}

In order to prove Theorem~\ref{fusilli} we will need the following simple result:
\begin{lemma}\label{transport-selection}
	Let $\Bi \in \boldsymbol{L}^1_\loc(I\times \R^d \times \R)$
	be such that $\dive_x \Bi(t,x,z) = 0$ for a.e. $(t,z)\in I \times \R$.
	Let $\rho \in L^1_\loc(I\times \R^d \times \R)$
	be such that
	$\rho \Bi \in \boldsymbol{L}^1_\loc(I\times \R^d \times \R)$.
	Suppose that $\rho$ solves
	\begin{equation}
		\d_t \rho + \dive_x (\rho \Bi) + \d_z \rho = 0, \qquad (t,x,z)\in I \times \R^d \times \R.
	\end{equation}
	Then for any $f\in L^\infty(\R)$ with compact support the function
	\begin{equation}
		\eta(t,x,z) := f(z-t) \rho(t, x, z)
	\end{equation}
	solves
	\begin{equation}
		\label{eta-CE}
		\d_t \eta + \div_x (\eta \Bi) + \d_z \eta = 0.
	\end{equation}
\end{lemma}

In particular, when $\Bi$ does not depend on $t$, the following result is a direct consequence of Lemma~\ref{transport-selection}: 

\begin{corollary}\label{transport-dim-extension}
	Let $\rho\in L^\infty(\R; L^2(\R^d))$ be a solution of \eqref{CE}
	with $\v \in L^\infty(\R; \boldsymbol{L}^2(\R^d))$
	such that $\dive_x \vi(z,x) = 0$ for a.e. $z\in \R$.
	Then for any $f\in L^\infty(\R)$ with compact support the function
	\begin{equation}
		\eta(t,x,z) := f(z-t) \rho(z, x)
	\end{equation}
	solves \eqref{eta-CE} with $\Bi(t,x,z) = \vi(z,x)$.
\end{corollary} 

\begin{proof}[Proof of Lemma~\ref{transport-selection}]
	Without loss of generality we may assume that $f\in C^1_c(\R)$.
	When $f$ is not smooth, one can approximate $f$ with functions from $C^1_c(\R)$ and pass to the limit in the distributional formulation of the continuity equation using dominated convergence.
	
	Since $F(t,z):=f(z-t)$ solves $\d_t F + \d_z F = 0$ and $\nabla_x F = 0$, by Leibniz rule
	for any $\Phi \in C^1_c(\R^{d+1})$ we have
	\begin{equation}
		(\d_t + \Bi \cdot \nabla_x + \d_z)(\Phi \cdot F) = F \cdot(\d_t\Phi + \Bi \cdot \nabla_x\Phi + \d_z \Phi).
	\end{equation}
	Then, denoting $\fhi(t,x,z):= f(z-t) \Phi(t,x,z)$, we have
	\begin{multline}
		\iiint_{\R\times\R^d\times \R}(\d_t \fhi(\cdot,z) + \Bi \cdot \nabla_x \fhi(\cdot,z) + \d_z \fhi(\cdot,z)) \rho(t,x,z) \, dx\,dt\,dz
		=\\
		\iiint_{\R\times\R^d\times \R} (\d_t \Phi + \Bi \cdot \nabla_x \Phi + \d_z \Phi) f(z-t) \rho(t,x,z) \, dt \, dx\, dz.
	\end{multline}
	The left-hand side is zero by the definition of weak solution of the continuity equation.
\end{proof}

\begin{proof}[Proof of Theorem~\ref{fusilli}]
	{\spaceskip=0.3em plus 0.4em minus 0.05em%
	The first building block of the construction is the \emph{steady whirlpool} ${\vv \colon \R^2 \to \R^2}$ given by}
	\begin{equation}
		\vv(x) = \nabla^\perp f(x),
		\qquad\text{where}\qquad f(x) = \frac{1}{2} \min\left(\|x\|_\infty^2, \frac{1}{4}\right).
		\label{whirlpool}
	\end{equation}
	Let $\rho$ denote the solution of the Cauchy problem for the continuity equation
	with the vector field $\vv$
	and the initial condition $\rho|_0 = \1_{[-1/2,1/2]^2} \cdot \sigma$, where
	\begin{equation}
		\label{plus-minus-one}
		\sigma(x) = \sign(\sin (\pi x_1)).
	\end{equation}
	Then $\vv$ has the following \emph{stirring property}: $\rho|_4(x)=\rho_0(-x)$ for all $x\in [-1/2,1/2]^2$
	(see Figure~\ref{whirlpool-stirring-fig}).
	
	Furthermore, it is easy to see that the stream function $f$ of $\vv$ has the weak Sard property; we will use this observation later.
	
	\begin{figure}[h!]
		\begin{tikzpicture}[scale=3,every node/.style={scale=1}]
	\foreach \x in {1/6,2/6,3/6} {
		%		\draw[-latex] (\x,-\x) -- (\x,\x) -- (-\x, \x) -- (-\x, -\x) -- cycle;
		\draw[-latex] (\x,-\x) -- (\x,\x);
		\draw[-latex] (\x,\x) -- (-\x, \x);
		\draw[-latex] (-\x, \x) -- (-\x, -\x);
		\draw[-latex] (-\x, -\x) -- (\x,-\x);
	}
	\tikzmath{\a=1/2;\b=\a*4/3;}
	\draw[-latex] (-\b,0) -- (\b,0) node[above] {$x_1$};
	\draw[-latex] (0,-\b) -- (0,\b) node[right] {$x_2$};
	\node[below right] at (1/2,0) {$\frac12$};
	\draw[fill] (1/2, 0) circle (.01) node[below right] {$\frac12$};
	\draw[fill] (0, 1/2) circle (.01) node[above left] {$\frac12$};
	\draw[dashed] (-\a,-\a) -- (\a, \a);
	\draw[dashed] (-\a,\a) -- (\a, -\a);
\end{tikzpicture}
		\tdplotsetmaincoords{50}{20}%
\begin{tikzpicture}[scale=2,tdplot_main_coords]
	\tdplotsetrotatedcoords{0}{90}{90}
	\begin{scope}[tdplot_rotated_coords]
		\tikzmath{
			\a=1/2;
			\b=\a*5/3;
			function xxx(\t) {
				%			return max(-1/2,min(1/2,-2*asin(cos(pi*\t/4)/pi)));
				return max(-1/2,min(1/2,-asin(sin(90*\t/2))/90));
			};
			function yyy(\t) {
				%			return max(-1/2,min(1/2,2*asin(cos(pi*\t/4))/pi));
				return max(-1/2,min(1/2,asin(cos(90*\t/2))/90));
			};
		}
		\coordinate (O) at (0,0,0);
		\foreach \s in {0,...,9} {
			\tikzmath{
				\dt=4/9;
				\t=\s*\dt;
				\xa=xxx(\t);
				\ya=yyy(\t);
				\xb=xxx(ceil(\t));
				\yb=yyy(ceil(\t));
				\xc=xxx(ceil(\t) + 1);
				\yc=yyy(ceil(\t) + 1);
				\xd=xxx(ceil(\t) + 2);
				\yd=yyy(ceil(\t) + 2);
				\xe=xxx(ceil(\t) + 3);
				\ye=yyy(ceil(\t) + 3);
				\xf=xxx(\t+4.0);
				\yf=yyy(\t+4.0);
			}
			\fill[red!40, opacity=0.9] (\xa,\ya,\t) -- (\xb,\yb,\t) -- (\xc,\yc,\t) -- (\xd,\yd,\t) -- (\xe,\ye,\t) -- (\xf,\yf,\t) -- cycle;
			\tikzmath{
				\p=\t+4;
				\xa=xxx(\p);
				\ya=yyy(\p);
				\xb=xxx(ceil(\p));
				\yb=yyy(ceil(\p));
				\xc=xxx(ceil(\p) + 1);
				\yc=yyy(ceil(\p) + 1);
				\xd=xxx(ceil(\p) + 2);
				\yd=yyy(ceil(\p) + 2);
				\xe=xxx(ceil(\p) + 3);
				\ye=yyy(ceil(\p) + 3);
				\xf=xxx(\p+4);
				\yf=yyy(\p+4);
			}
			\fill[green!40, opacity=0.9] (\xa,\ya,\t) -- (\xb,\yb,\t) -- (\xc,\yc,\t) -- (\xd,\yd,\t) -- (\xe,\ye,\t) -- (\xf,\yf,\t) -- cycle;
			\ifthenelse{\equal{\s}{9}}{}{%
				\draw (0,0,\t) -- (0,0,\t+\dt);
			}
		}
		\draw[-latex] (-\b,0,0) -- (\b,0,0) node[above right] {$x_1$};
		\draw[-latex] (0,-\b,0) -- (0,\b,0) node[above] {$x_2$};
		\draw[-latex] (0,0,4) -- (0,0,4.5) node[right] {$t$};
		\node at (0,0,4) {$\bullet$};
		\node[above right] at (0,0,4) {$4$};
	\end{scope}
\end{tikzpicture}
		\caption{
			Steady whirlpool $\vv$ given by \eqref{whirlpool} and the solution of the associated continuity
			equation with the initial data \eqref{plus-minus-one}.}
			\label{whirlpool-stirring-fig}
	\end{figure}
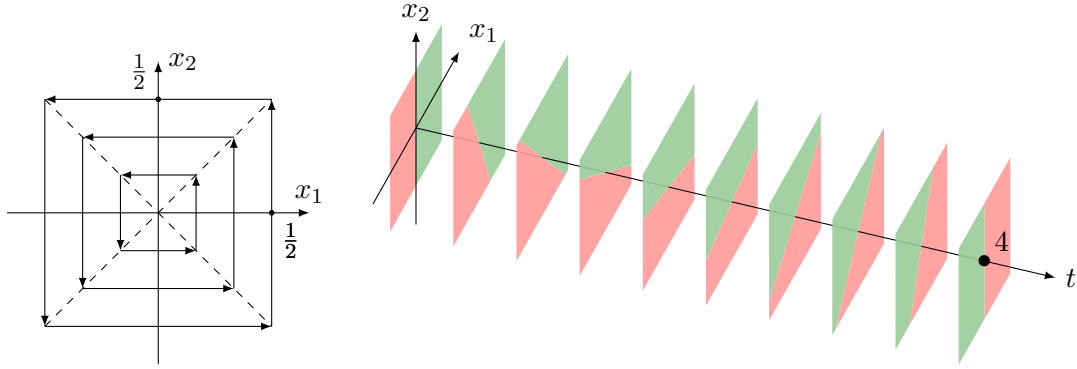
	
	Let $E:=[-1,1]^2$. Let $\yu \colon \R^2 \to \R^2$ be a $2$-periodic (with respect to each variable) vector field
	such that for all $(x_1, x_2) \in E$
	\begin{equation}
		\yu(x_1, x_2) := \vv(x_1, x_2 - \tfrac12) + \vv(x_1, x_2 + \tfrac12).
	\end{equation}
	(See Figure~\ref{periodic-whirlpool-fig}.)
	
	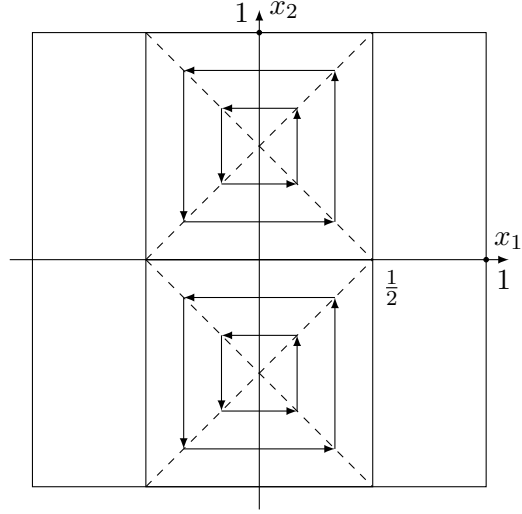
\begin{figure}[h!]
		\begin{tikzpicture}[scale=3,every node/.style={scale=1}]
	\tikzmath{\x=1;}
	\draw[] (\x,-\x) -- (\x,\x) -- (-\x, \x) -- (-\x, -\x) -- cycle;
	\begin{scope}[shift={($(0,1/2)$)}]
		\tikzmath{\x=1/2;}
		\draw[] (\x,-\x) -- (\x,\x) -- (-\x, \x) -- (-\x, -\x) -- cycle;
		\foreach \x in {1/6,2/6} {
			\draw[-latex] (\x,-\x) -- (\x,\x);
			\draw[-latex] (\x,\x) -- (-\x, \x);
			\draw[-latex] (-\x, \x) -- (-\x, -\x);
			\draw[-latex] (-\x, -\x) -- (\x,-\x);
		}
		\draw[dashed] (-\x,-\x) -- (\x, \x);
		\draw[dashed] (-\x,\x) -- (\x, -\x);
	\end{scope}
	\begin{scope}[shift={($(0,-1/2)$)}]
		\tikzmath{\x=1/2;}
		\draw[] (\x,-\x) -- (\x,\x) -- (-\x, \x) -- (-\x, -\x) -- cycle;
		\foreach \x in {1/6,2/6} {
			\draw[-latex] (\x,-\x) -- (\x,\x);
			\draw[-latex] (\x,\x) -- (-\x, \x);
			\draw[-latex] (-\x, \x) -- (-\x, -\x);
			\draw[-latex] (-\x, -\x) -- (\x,-\x);
		}
		\draw[dashed] (-\x,-\x) -- (\x, \x);
		\draw[dashed] (-\x,\x) -- (\x, -\x);
	\end{scope}
	\tikzmath{\a=1;\b=\a*1.1;}
	\draw[-latex] (-\b,0) -- (\b,0) node[above] {$x_1$};
	\draw[-latex] (0,-\b) -- (0,\b) node[right] {$x_2$};
	\node[below right] at (1/2,0) {$\frac12$};
	\draw[fill] (1, 0) circle (.01) node[below right] {$1$};
	\draw[fill] (0, 1) circle (.01) node[above left] {$1$};
\end{tikzpicture}
		\caption{The set $E$ and the vector field $\yu$.}
		\label{periodic-whirlpool-fig}
	\end{figure}
	Let $\rho$ denote the solution of the Cauchy problem for the continuity equation
	with the vector field $\yu$
	and the initial condition $\rho|_0 = \1_{E} \cdot \sigma$, where
	$\sigma$ is given by \eqref{plus-minus-one}.
	By linearity of \eqref{CE} this Cauchy problem can be easily reduced
	to the Cauchy problem for \eqref{CE} with the vector field $\vv$.
	Then, by the stirring property of $\vv$, one can conclude that $\yu$
	has the following \emph{stirring property}:
	\begin{equation}
		\label{yu-stirring}
		\rho|_{4}(x) = \1_E(x) \cdot \sigma(2 x_1, x_2).
	\end{equation}

	Let $T_0:=0$ and for any $n\in \N$ let
	\begin{equation}
		T_n := 4 \sum_{k=0}^{n-1} 2^{-k}
		\qquad\text{and}\qquad
		T_\infty:= \lim_{n\to \infty} T_n = 8.
	\end{equation}
	Let us define the map $\bi\colon \R^3 \to \R^2$ by
	\begin{equation}
		\bi(x_1, x_2, z) :=
		\begin{cases}
			\1_{E} \cdot \yu_n(x_1,x_2), & \text{if } z\in [T_{n}, T_{n+1}) \text{ for some } n\in \Z_+;\\ 
			% TODO: define \Z_+ in Notation
			0 & \text{otherwise}
		\end{cases}
	\end{equation}
	(see Figure~\ref{b-fig}), where
	\begin{equation}
		\yu_n(x_1, x_2) := \yu(2^{n}(x_1+1) - 1, 2^{n} x_2).
	\end{equation}
	
	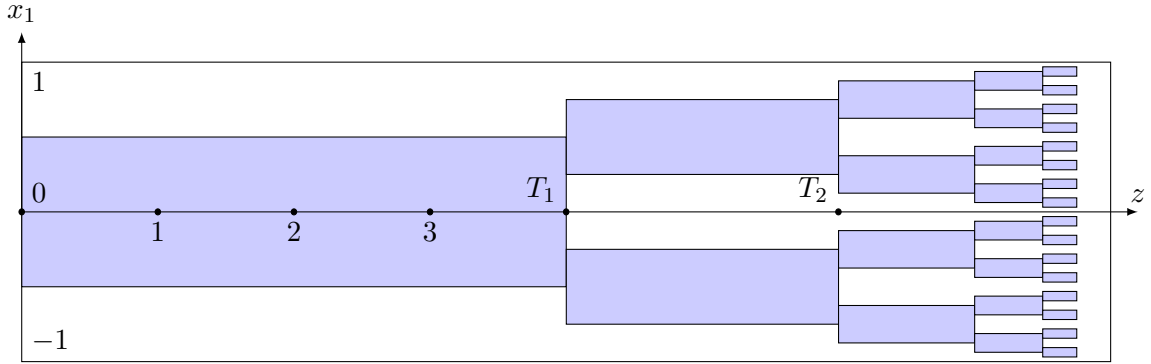
\begin{figure}[h!]
		\begin{tikzpicture}[scale=1.8,every node/.style={scale=1}]
\begin{scope}[yscale=1.1]
\draw (0,-1) rectangle (8,1);
\foreach \n [
	evaluate=\n as \side using 2^(1-\n),
	evaluate=\side as \tt using 4*\side+\tt,
	remember=\tt as \tt (initially 0),
	evaluate=\n as \m using 2^(\n-1),
	remember=\m as \m
	] in {1,...,5} {
	\foreach \k in {1,...,\m} {
		\node (A) at (\tt,-1+2*\k*\side-\side-\side/2) {};% {$\bullet$};
		\draw[fill=blue!20] (A) rectangle ++(-4*\side,\side);
	}
	\ifthenelse{\n < 3}{%
		\draw[fill] (\tt,0) circle (0.02) node[above left] {$T_\n$};
	}{}
}
\draw[fill] (0,0) circle (0.02) node[above right] {$0$};
\draw[-latex] (0,0) -- (0,1.2) node[above] {$x_1$};
\draw[-latex] (0,0) -- (8.2,0) node[above] {$z$};
\foreach \n in {1,...,3} {
	\draw[fill] (\n,0) circle (0.02) node[below] {$\n$};
}
\node[below right] at (0,1) {$1$};
\node[above right] at (0,-1) {$-1$};
\end{scope}
\end{tikzpicture}
		\caption{The (support of the) vector field $\bi$.}
		\label{b-fig}
	\end{figure}
	
	Now let us consider the continuity equation in $\R^2$ with the (non-autonomous) vector field $\vi(t,x_1,x_2) = \bi(x_1, x_2, t)$.
	Let $\varrho$ denote the solution of the associated Cauchy problem with the initial condition $\varrho|_0 = \varrho_0$,
	where
	\begin{equation}
		\varrho_0 = \1_{E} \cdot \sigma,
	\end{equation}
	with $\sigma$ given by \eqref{plus-minus-one}.
	Using the stirring property of $\yu$ \eqref{yu-stirring} one can see that for all $n\in \Z_+$
	\begin{equation}
		\varrho|_{T_n}(x) = \1_E(x) \cdot \sigma(2^n x_1, x_2).
	\end{equation}
	The sequence $\{\varrho|_{T_n}\}$ converges to zero weakly in $L^2(\R^2)$ as $n\to \infty$.
	Hence $\varrho|_{T_\infty} = 0$ and therefore $\varrho(t, \cdot) = 0$ for almost all $t>T_\infty$.
	
	Extending $\varrho$ with the initial condition for negative times (see Remark~\ref{extension-with-zero}) we obtain a solution of \eqref{CE}
	with the following energy profile:
	\begin{equation}
		\label{Depauw-energu-profile}
		\int_{\R^2} \varrho^2(t,x) \, dx = \begin{cases}
			4, & t \le T_\infty, \\
			0, & t > T_\infty.
		\end{cases}
	\end{equation}
	By Corollary~\ref{transport-dim-extension} the function
	\begin{equation}
		\eta(t,x,z) = \1_{[0,1]}(z-t) \varrho(z,x) = \begin{cases}
			\varrho(z,x), & t\le z \le t+1,\\
			0, & \text{otherwise}
		\end{cases}
	\end{equation}
	solves the continuity equation in $\R^3$ with the autonomous vector field
	\begin{equation}
		\Bi(x,z) = \begin{pmatrix}
			b_1(x,z)\\
			b_2(x,z)\\
			1
		\end{pmatrix},
	\end{equation}
	where $b_1$ and $b_2$ are the components of $\bi$.
	
	By \eqref{Depauw-energu-profile} it is easy to see that
	\begin{equation}
		\label{energy-profile-eta}
		\int_{\R}\int_{\R^2} \eta^2(t, x, z) \, dx \, dz
		= \begin{cases}
			4 & t < T_\infty-1, \\
			4(T_\infty-t), & T_\infty-1 \le t < T_\infty \\
			0, & t \ge T_\infty.
		\end{cases}
	\end{equation}
	Since $\eta$ is a nontrivial solution and $\eta=0$ for $t>T_\infty$, 
	we have no uniqueness backward in time for the continuity equation with $\Bi$.
	And since the renormalization property implies conservation of energy (by Theorem~\ref{renormalized-solutions}),
	which in our case evidently fails by \eqref{energy-profile-eta}, $\Bi$ does not have the renormalization property.
	
	It remains to show that if $\rho\in L^\infty(I; L^2(\R^3))$ solves the Cauchy problem for the continuity equation with the vector field $\Bi$ and satisfies the initial condition $\rho|_0 = 0$, then $\rho =0$.
	
	We are going to prove that for any non-negative integer $n$ the solution $\rho$ vanishes in the halfspace
	$I\times H_n^-$, where $H_{n}^- := \{(x,z) \in \R^2 \times \R : z < T_n\}$. 
	Let us also denote $H_{n}^+ := \{(x,z) \in \R^2 \times \R : z \ge T_n\}$.
	
	By Lemma~\ref{transport-selection}
	for any $f\in C^1_c(\R)$ the function
	$\eta(t,x,z) = f(z-t) \rho(t,x,z)$ also solves the continuity equation
	and satisfies the initial condition $\eta|_0 = 0$.
	Observe that for all $t\in I$ the support of $z\mapsto f(z-t)$ is $t + \supp f$.
	
	First let us consider $n=0$. Then for any $\tau \in I$ for any $\psi \in C^1_c(\R)$
	with $\supp \psi \subset (-\infty, 0)$ we can take $f(z) := \psi(\tau + z)$
	and then for all $t < \tau$ the function $\eta(t, \cdot)$
	vanishes in the halfspace $H_{0}^+$.
	Then one can redefine $\Bi$ to $(0,0,1)$ in the halfspace $H_0^+$, and $\eta$ will
	still solve the continuity equation (for $t<\tau$).
	But then $\Bi \equiv (0,0,1)$ in $\R^3$,
	and by the classical theory $\eta$ vanishes. 
	By arbitrariness of $\psi$ (and $\tau$) we conclude that $\rho$ vanishes in $I\times H_0^-$.
	
	Since $\rho$ vanishes in $I \times H_0^-$, we can redefine in $H_0^-$ the vector field $\Bi$ as follows:
	\begin{equation}
		\Bi(x, z) := \Bi(x, \xi) = \begin{pmatrix} \yu(x) \\ 1\end{pmatrix},
	\end{equation}
	where the constant $\xi \in (T_0, T_1)$ is arbitrary (by the construction $\Bi$
	does not depend on $z$ when $z\in (T_{n}, T_{n+1})$ with a fixed $n\in \Z_+$).
	
	For any $\tau \in I$ for any $\psi \in C^1_c(\R)$
	with $\supp \psi \subset (-\infty, 0)$ we can take $f(z) := \psi(\tau + z)$
	and then for all $t < \tau$ the function $\eta(t, \cdot)$
	vanishes in the halfspace $H_{1}^+$. Then we can extend~$\Bi$ from $H_0^-$ (or, equivalently, from $H_1^-$)
	to $H_1^+$ in the same way. The resulting vector field does not depend on $z$,
	and the stream function of the corresponding vector field $\yu$
	has the weak Sard property (as we already observed for a single whirlpool).
	Hence by Theorem~\ref{uniqueness-forward-and-backward} the operator $\yu \cdot \nabla$ is essentially skew-adjoint. 
	By Proposition~\ref{skew-adjoint-extra-dimension}  the operator $\Bi \cdot \nabla$ (with the redefined vector field $\Bi$) is essentially skew-adjoint, hence we have uniqueness for the
	Cauchy problem for the associated continuity equation. 
	Hence $\eta$ vanishes.
	By arbitrariness of $\psi$ (and $\tau$) we conclude that $\rho$ vanishes in $I\times H_1^-$.
	
	Proceeding by induction, we conclude that for all $n\in \Z_+$ 
	the solution $\rho$ vanishes in $I\times H_n^-$.
	This means that $\rho$ vanishes in the halfspace $\{(t,x,z)\in I \times \R^2 \times \R : z\le T_\infty\}$.
	Then arguing as in the case $n=0$ we can redefine $\Bi$ to the constant vector $(0,0,1)$
	in the whole space and conclude that $\rho$ vanishes in $I\times \R^3$.
\end{proof}

\begin{remark}\label{uroboros}
	In the statement of Theorem~\ref{fusilli} we do not claim that the vector field $\vi$
	has compact support. In fact the constructed vector field $\Bi$ does not have compact support.
	However the construction can be modified in order to achieve that, in addition, the resulting vector field~$\vi$ has compact support (see Figure~\ref{uroboros-fig}).
	Since the renormalization property fails for $\vi$, the corresponding operator $\vi \cdot \nabla$
	is not essentially skew adjoint.
\end{remark}

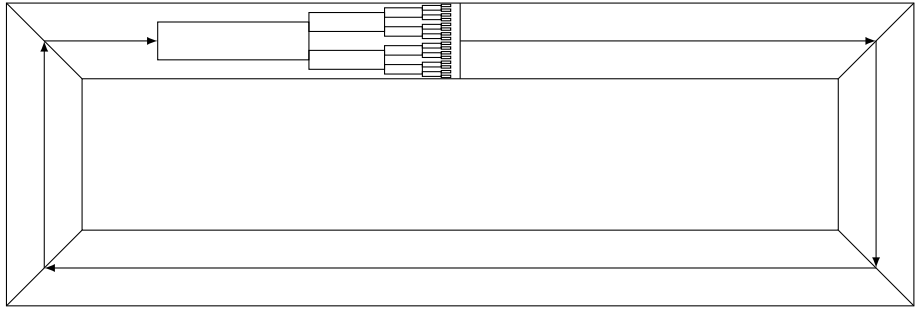
\begin{figure}[h!]
	\begin{tikzpicture}[scale=0.5,every node/.style={scale=1}]
	\tikzmath{
		\x=-8;
		\y=3;
		\aa=12;
		\b=4;
	}
	\draw (-\aa,-\b) rectangle (\aa,\b);
	\draw (-\aa+2,-\b+2) rectangle (\aa-2,\b-2);
	\draw (-\aa,-\b) -- (-\aa+2,-\b+2);
	\draw (-\aa,\b) -- (-\aa+2,\b-2);
	\draw (\aa,\b) -- (\aa-2,\b-2);
	\draw (\aa,-\b) -- (\aa-2,-\b+2);
	\draw[-latex] (0, \b-1) -- (\aa-1,\b-1);
	\draw[-latex] (\aa-1,\b-1) -- (\aa-1,-\b+1);
	\draw[-latex] (\aa-1,-\b+1) -- (-\aa+1,-\b+1);
	\draw[-latex] (-\aa+1,-\b+1) -- (-\aa+1,\b-1);
	\draw[-latex] (-\aa+1,\b-1) -- (\x,\y);
	\begin{scope}
%		\draw(\x,\y-1) rectangle (\x+8,\y+1);
		\draw(\x+8,\y+1) -- (\x+8,\y-1);
		\foreach \n [
		evaluate=\n as \side using 2^(1-\n),
		evaluate=\side as \tt using 4*\side+\tt,
		remember=\tt as \tt (initially 0),
		evaluate=\n as \m using 2^(\n-1),
		remember=\m as \m
		] in {1,...,5} {
			\foreach \k in {1,...,\m} {
				\node (A) at (\x+\tt,\y-1+2*\k*\side-\side-\side/2) {};% {$\bullet$};
				\draw (A) rectangle ++(-4*\side,\side);
			}
		}
	\end{scope}
\end{tikzpicture}
	\caption{A sketch of the construction discussed in Remark~\ref{uroboros}.}
	\label{uroboros-fig}
\end{figure}

\begin{remark}\label{sum-of-skadj-is-not}
	Let $\boldsymbol{w}$ denote the vector field obtained from $\vi$ in the previous remark by replacing the vector field $\bi$ in the construction of $\Bi$ with zero. 
	Let $\boldsymbol{u} := \vi - \boldsymbol{w}$.
	Then $(\vi \cdot \nabla) = (\boldsymbol{u} \cdot \nabla) + (\boldsymbol{w} \cdot \nabla)$
	and one can prove that the operators $(\boldsymbol{u} \cdot \nabla)$ and $(\boldsymbol{w} \cdot \nabla)$
	(defined on $C^1_c(\R^3)$) are essentially skew adjoint, while their sum is not (by Remark~\ref{uroboros}).
\end{remark}

\section{Stability of solutions}

% !TeX root = skew-adjoint.tex

In this section we show that some of the stability results established in \cite[Theorem 11.4]{DiPernaLions1989}
can be easily recovered using the theory of semigroups.

\begin{theorem}
	Suppose that $\{\vi_n\} \subset \boldsymbol{L}^2(\R^d)$ is a sequence of divergence-free vector fields
	and $\vi_n \to \vi$ in $\boldsymbol{L}^2(\R^d)$ as $n\to \infty$.
	Suppose that the operators $A_n \rho := (\vi_n \cdot \nabla) \rho$ and $A\rho := (\vi \cdot \nabla) \rho$ on $L^2(\R^d)$ with the domains $D(A_n) = D(A) = C^1_c(\R^d)$ are essentially skew-adjoint.
	Suppose that $\theta_n \to \theta$ in $L^2(\R^d)$ and let $\rho_n \in L^\infty(0,T; L^2(\R))$ denote the sequence of solutions of \eqref{CE} associated with $\vi_n$ such that $\rho_n|_0 = \theta_n$. 
	Then $\rho_n \to \rho$ in $L^\infty(0,T; L^2(\R^d))$ as $n\to \infty$, where $\rho$ is the solution of \eqref{CE} with $\rho|_0 = \theta$.
\end{theorem}

\begin{proof}
	Let $H:= L^2(\R)$.
	Let $\overline{A}$ and $\overline{A_n}$ denote the closures of $A$ and $A_n$ respectively.
	Let $S_t$ and $S^n_t$ denote the $c_0$-groups generated by $-\overline{A}$ and $-\overline{A_n}$ respectively.
	In view of Corollary~\ref{strongly-cont-version-is-given-by-group} it is sufficient to prove that $S^n_t \theta_n$ converges in $C([0,T]; H)$ to $S_t \theta$ as $n\to\infty$.
	Since
	\begin{equation*}
		S^n_t \theta_n - S_t \theta = S^n_t (\theta_n -\theta) + (S^n_t \theta - S_t \theta)
	\end{equation*}
	and $\|S^n_t (\theta_n - \theta)\| \le \|\theta_n - \theta\|\to 0$ as $n\to \infty$,
	it remains to show that the second term in the above equality tends to zero.
	
	Let $E$ denote the identity operator on $H$.
	We are going to prove that for any $\lambda \in \R\setminus \{0\}$
	the resolvents $(\lambda E - \overline{A_n})^{-1}$ converge strongly to $(\lambda E - \overline{A})^{-1}$ as $n \to \infty$.
	Then by the Trotter--Kato theorem $\|S^n_t \theta - S_t \theta\| \to 0$ uniformly on $[0,T]$ as $n\to \infty$.
	
	Suppose that $f \in H$, $u_n \in D(\overline{A_n})$ and $u \in D(\overline{A})$ are such that
	\begin{align}
		\lambda u_n - \overline{A_n} u_n &= f, \notag\\
		\lambda u - \overline{A} u &= f \label{Alambda-rho=f}.
	\end{align}
	Since the operator $\overline{A}$ is skew-adjoint, we have
	\begin{equation}
		\|f\|^2 = \|(\lambda E - \overline{A})u\|^2 =  \lambda^2\|u\|^2 + \|\overline{A} u\|^2.
	\end{equation}
	Hence $|\lambda|\|u\| \le \|f\|$ and similarly for all $n\in \N$ we have $|\lambda|\|u_n\|\le \|f\|$.
	Then by sequential weak compactness of the closed unit ball in $H$ there exists a subsequence $\{u_{n_k}\}$
	which converges weakly to $w \in H$. 
	On the other hand, since the operators $\overline{A_n}$ are skew-adjoint, for all $n\in \N$ we have
	$\lambda u_n + A_n^* u_n = f$ and consequently for any $\fhi \in C^1_c(\R^d)$
	\begin{equation}\label{test-resolvent}
		\lambda (\fhi, u_n) + ({A_n} \fhi, u_n) = (\fhi, f).
	\end{equation}
	Since $\vi_n \to \vi$ strongly in $\boldsymbol{L}^2(\R^d)$ as $n\to \infty$, it follows that $A_n \fhi \to A \fhi$ strongly in $H$.
	Since $u_{n_k} \to w$ weakly in $H$ as $k\to \infty$, it follows
	that $({A_{n_k}} \fhi, u_{n_k}) \to ({A} \fhi, w)$.
	So substituting $n=n_k$ in \eqref{test-resolvent} and passing to the limit as $k\to \infty$ we obtain that
	\begin{equation*}
		\lambda (\fhi, w) + ({A}\fhi, w) = (\fhi, f).
	\end{equation*}
	By arbitrariness of $\fhi\in C^1_c(\R^d)$ it follows that $w \in D(A^*)$ and $\lambda w + A^* w = f$, or in other words
	\begin{equation}\label{Alambda-eta=f}
		\lambda w - \overline{A} w = f.
	\end{equation}
	 Subtracting \eqref{Alambda-rho=f} from \eqref{Alambda-eta=f}, we obtain
	\begin{equation*}
		\lambda(w - u) - \overline{A} (w - u) = 0,
	\end{equation*}
	hence $w = u$ by invertibility of $(\lambda E - \overline{A})$.
	Since the arguments above can be applied to any subsequence of $u_n$, it follows that the whole sequence $u_n$ converges weakly to $u$ as $n\to \infty$.
	It remains to prove that this convergence is strong.
	
	Since $\overline{A}$ is skew-adjoint, $(u, \overline{A} u) = 0$.
	Then by \eqref{Alambda-rho=f} we have
	\begin{equation}\label{norm-as-inner-prod}
		\lambda\|u\|^2 = (f, u)
		\quad\text{and similarly}\quad
		\lambda \|u_n\|^2 = (f, u_n).
	\end{equation}
	Since $u_n \to u$ weakly as $n\to \infty$, it follows from \eqref{norm-as-inner-prod} that $\|u_n\|^2 \to \|u\|^2$.
	It is well-known that, together with weak convergence of $u_n$ to $u$, this implies that $u_n$ converges to $u$ strongly in~$H$.
\end{proof}

Let us mention some results which are related to the assumption that the operator $A$, corresponding to the limit vector field $\vi$, is essentially skew-adjoint.
In \cite{CCS2019} it was proved that there exist an autonomous divergence-free vector field $\vi \in \boldsymbol{L}^{4/3}(\R^3)$ and a sequence of smooth autonomous divergence-free vector fields $\vi_n$ converging to $\vi$, such that the sequence of the corresponding approximate solutions $\rho_n$ has two distinct limits (in $L^1_{\loc}$).
In \cite{DLG2021} it was proved that for $d\ge 2$ there exists an autonomous compactly supported bounded divergence-free vector field $\vi$ and initial data $\rho_0 \in C^\infty_c(\R^{d+1})$ such that there exist
two sequences of compactly supported divergence-free smooth vector fields $\vi_n$ and $\vv_n$
converging to $\vi$ strongly in $L^1$ as $n\to \infty$ such that the corresponding solutions $\rho_n$ and $\theta_n$ of the continuity equation for $\vi_n$ and $\vv_n$ converge respectively to two distinct solutions of the continuity equation for $\vi$ with the initial condition $\rho_0$.
On the other hand, in \cite{Pitcho2024} it was proved that under some additional regularity assumptions on the non-autonomous divergence-free vector field $\vi$ (which are in general too weak for the renormalization property) the solutions of the continuity equation associated with the standard mollifications of $\vi$ converge in the sense of distributions to a unique limit.

\section{Acknowledgements}\label{sec:ack}
The work of N.G. and K.Z. was supported by the RSF project 24-21-00315.

\appendix
\section{Uniqueness of measure-preserving flows}\label{a:uniqueness-of-mpf}

%\ngnote{На этот раздел можно будет сослаться из введения при обсуждении истории вопроса.}

Let $p\ge 1$.

\begin{definition}
	A family $\{X_t\}_{t\in \R}$ of maps $X_t\colon \R^d \to \R^d$
	is called a \emph{group of measure-preserving transformations}, if
	\begin{enumerate}
		\item[1.] the map $t\mapsto (X_t - \operatorname{id})$ belongs to $C(\R;L^p(\R^d; \R^d))$;
		\item[2.] for all $t,s\in \R$ we have $X_s \circ X_t = X_{s+t}$ a.e.;
		\item[3.] for all $t\in \R$ we have $(X_t)_\# \L^d = \L^d$.
		%\item[4.] for all $t\in \R$ we have
		%\begin{equation}
		%	X_t(y) = y + \int_0^t \v(X_s(y)) \, ds \label{integral-curves}
		%\end{equation}
		%in the following sense: for any $\fhi \in C_c^\infty(\R^d)$
		%\begin{equation*}
		%	\int_{\R^d} X_t(y) \fhi(y) \, dy = \int_{\R^d} y \fhi(y) \, dy
		%	+ \int_0^t \int_{\R^d} \vi(y) \fhi(X_{-s}(y)) \, dy \, ds.
		%\end{equation*}
		%% \\ (in $\mathscr D'(\R^d_y)$);
	\end{enumerate}
\end{definition}

\begin{lemma}\label{comp-with-measure-preserving-transformations}
	Let $\{X_t\}_{t\in \R}$ be a group of measure-preserving transformations.
	Then for any $f \in L^p(\R^d)$ and for all $t\in \R$ we have $f\circ X_t \in L^p(\R^d)$ and $\|f\circ X_t\|_p = \|f\|_p$.
	Furthermore, the map $t\mapsto f\circ X_t$ belongs to $C(\R; L^p(\R^d))$.
\end{lemma}

\begin{proof}
	First let $f\in C_c^\infty(\R^d)$ and let $t\in \R$. 
	Then $f\circ X_t$ is measurable and $\|f\circ X_t\|_p = \|f\|_p$.
	Hence $f\circ X_t \in L^p(\R^d)$.
	Furthermore,
	\begin{equation*}
		\|f\circ X_t - f\circ X_s\|_p \le \|\nabla f\|_\infty \|X_t - X_s\|_p
	\end{equation*}
	and it follows that $t\mapsto f\circ X_t$ belongs to $C(\R; L^p(\R^d))$.
	
	Now let us suppose that $f\in L^p(\R^d)$.
	Let us choose a sequence $f_n \in C_c^\infty(\R^d)$ such that $f_n \to f$ a.e. and $\|f_n - f\|_p \to 0$ as $n\to \infty$.
	Since $X_t$ preserves the Lebesgue measure, for any $t\in \R$ we have $f_n \circ X_t \to f\circ X_t$ a.e. as $n\to \infty$.
	Hence for all $t\in \R$ the function $f\circ X_t$ is measurable. 
	Furthermore, for all $t\in \R$ we have
	\begin{equation*}
		\|f\circ X_t - f_n \circ X_t\|_p = \|f - f_n\|_p,
	\end{equation*}
	hence the maps $t\mapsto f_n \circ X_t$ converge in $L^p(\R^d)$ uniformly on $\R$ to the map $t\mapsto f \circ X_t$ as $n \to \infty$.
	Hence $t\mapsto f \circ X_t$ belongs to $C(\R; L^p(\R^d))$.
\end{proof}

Let $\vi \colon \R^d \to \R^d$ be a compactly supported
vector field and suppose that $\vi \in \boldsymbol{L}^p(\R^d)$.
Then by the previous lemma (applied to each component of $\vi$) we have $[t\mapsto \vi \circ X_t] \in C(\R; L^p(\R^d; \R^d))$.
Hence the latter map is locally integrable (e.g. in the Riemann sense).

\begin{definition}\label{def-measure-preserving-flow}
	The group $\{X_t\}_{t\in \R}$ of measure-preserving transformations is called a \emph{measure-preserving flow of $\vi$}, if
	for all $t\in \R$ we have
	\begin{equation}
		X_t - \operatorname{id} = \int_0^t \v \circ X_s \, ds. \label{integral-curves}
	\end{equation}
\end{definition}

The following proposition appears to be well-known.

\begin{proposition}\label{unique-flow}
	Suppose that $p=2$.
	Let $A_0 \colon L^2(\R^d) \supset D(A_0) \to L^2(\R^d)$ be given by
	$A_0 \rho = \v \cdot \nabla \rho$ with the domain $D(A_0) = C_c^\infty(\R^d)$.
	If $A_0$ is essentially skew-adjoint, then $\v$
	cannot have two different measure-preserving flows.
\end{proposition}

\begin{proof}
	By Definition~\ref{def-measure-preserving-flow} we have $X_0 = \operatorname{id}$.
	By Lemma~\ref{comp-with-measure-preserving-transformations} any measure-preserving flow $X_t$ of the vector field $\v$ induces a $c_0$-group of isometries $S_t \colon L^2(\R^d) \to L^2(\R^d)$ given by
	\begin{equation}
		S_t f := f \circ X_t, \qquad f \in L^2(\R^d).
	\end{equation}
	
	%For any $t,s\in \R$ (such that $s<t$) the functions $X_t$ and $X_s$
	%in general do not belong to $L^2(\R^d)$, unlike their difference:
	%\begin{equation}
	%\label{flow-diff-estimate}
	%\|X_t - X_s\| = \left\| \int_s^t \v\circ X_\tau \, d\tau \right\|
	%\le \int_s^t \|\v \circ X_\tau\| \, d\tau = (t-s) \|\v\|.
	%\end{equation}
	%
	%It is easy to see that $S_t$ is a $c_0$-group.
	%Indeed, for any $f\in C^\infty_c(\R^d)$
	%by the mean value theorem $|f(X_t(x)) - f(X_s(x))| \le C |X_t(x) - X_s(x)|$,
	%where $C=\|\nabla f\|_\infty$.
	%Then by \eqref{flow-diff-estimate}
	%\begin{equation*}
	%\|f \circ X_t - f\circ X_s\| \le C \|X_t - X_s\| \le C (t-s) \|\v\|.
	%\end{equation*}
	%In the more general case when $f \in L^2(\R^d)$
	%one can approximate $f$ with $\widetilde{f} \in C^\infty_c(\R^d)$
	%and use the estimate above
	%together with the inequality $\|f \circ X_t - \widetilde f\circ X_t\| = \|f - \widetilde{f}\|$.
	
	Let $G$ denote the generator of the group $S_t$.
	We claim that
	\begin{equation}
		\label{universal-subset-of-domain-of-generator}
		A_0 \subset G.
%		C_c^\infty(\R^d) \subset D(G).
	\end{equation}
	It suffices to prove that
	for any $\fhi \in C_c^\infty(\R^d)$ we have
	\begin{equation}
		\label{derivative-of-St}
		\frac{\fhi \circ X_t - \fhi}{t} \to \v \cdot \nabla \fhi
	\end{equation}
	in $L^2(\R^d)$ as $t\to 0$.
	
	By \eqref{integral-curves} for a.e. $x\in \R^d$ the map $t\mapsto X_t(x)$ is absolutely continuous with 
	\[\d_t(X_t(x)) = \v(X_t(x))\]
	for a.e. $t$.
	Then $t\mapsto \fhi \circ X_t$ is absolutely continuous with $\d_t \fhi \circ X_t = (\v \cdot \nabla \fhi) \circ X_t$ for a.e. $t$.
	Hence $\fhi \circ X_t - \fhi = \int_0^t (\v \cdot \nabla \fhi) \circ X_s \, ds$. Since $S_t$ is a $c_0$-group, for any $\eps > 0$
	there exists $\delta > 0$ such that for all $s$ with $|s|< \delta$
	we have $\|(\v \cdot \nabla \fhi) \circ X_s - \v \cdot \nabla \fhi\|_p < \eps$.
	Then for all $t$ with $|t| < \delta$ we have
	\begin{equation*}
		\left\|\frac{\fhi \circ X_t - \fhi}{t} - \v \cdot \nabla \fhi\right\|_p
		\le \frac{1}{t}\int_0^t \|(\v \cdot \nabla \fhi) \circ X_s - \v \cdot \nabla \fhi\|_p \, ds
		\le \frac{1}{t}\int_0^t \eps \, ds = \eps.
	\end{equation*}
	Hence \eqref{derivative-of-St} holds, and \eqref{universal-subset-of-domain-of-generator} follows.
	
	By Stone's theorem $G$ is skew-adjoint.
	By \eqref{universal-subset-of-domain-of-generator}
	we have $A_0 \subset G$.
	Hence $A \subset G$, where $A = \overline{A_0}$ is the closure of $A_0$.
	Then $-G = G^* \subset A^* = - A$, since $A$ is skew-adjoint.
	Thus $A \subset G$ and $G \subset A$, hence $G=A$.
	Hence the group $S_t$ is unique,
	and this implies uniqueness of the flow $X_t$.
	Indeed, if $\widetilde{X}_t$ is another such flow, then for any $f\in L^2(\R^d)$
	we have $f \circ X_t = f \circ \widetilde{X}_t$.
	Substituting $f(x) = x_i \cdot \1_{B_r(0)}(x)$ with  arbitrary $i\in \{1, ..., d\}$ and $r \in \N$ we conclude that $X_t = \widetilde{X}_t$ a.e.
\end{proof}

\bibliographystyle{alpha}
\bibliography{references}

\end{document}